\documentclass[10pt,reqno]{amsart}

\usepackage{a4wide}
\usepackage[all]{xy}
\usepackage{upgreek}
\usepackage{dsfont} 
\usepackage{graphics}
\usepackage[utf8]{inputenc}
\usepackage{t1enc}
\usepackage{amsmath}
\usepackage{amsfonts}
\usepackage{amssymb}
\usepackage{bbm}
\usepackage{mathrsfs}
\usepackage{enumerate}

\usepackage{lmodern}
\newcommand{\ov}{\overline}
\newcommand{\id}{\op{Id}}
\newcommand{\I}{\mathds{1}}
\newcommand{\md}{\operatorname{d}\!}
\newcommand{\cst}{\ensuremath{\mathrm{C}^*}}

\newcommand{\vp}{\varphi}
\newcommand{\whG}{\widehat{\GG}}
\newcommand{\whH}{\widehat{\HH}}

\newcommand{\mc}{\mathcal}

\newcommand{\msf}{\mathsf}

\newcommand{\mscr}{\mathscr}

\newcommand{\op}[1]{\operatorname{#1}}

\newcommand{\tp}{\!\!
{\scriptstyle
\text{
\raisebox{0.8pt}{
\textcircled{\raisebox{-1.7pt}{$\,\top$}}
} 
} 
} 
\!\!}

\newcommand{\stp}{\!\!\!
{\scriptscriptstyle
\text{
\raisebox{0.5pt}{
\textcircled{\raisebox{-1.2pt}{$\,\top$}}
} 
} 
} 
\!\!\!}

\newcommand{\ismaa}[2]{\langle#1\,|\,#2\rangle}

\newcommand{\Linf}{\operatorname{L}^{\infty}(\GG)}
\newcommand{\LdG}{\operatorname{L}^{2}(\GG)}
\newcommand{\IrrG}{\Irr(\GG)}
\newcommand{\SU}{\mathrm{SU}}

\newcommand{\oon}{\operatorname}

\newcommand{\GG}{\mathbb{G}}
\newcommand{\NN}{\mathbb{N}}
\newcommand{\ZZ}{\mathbb{Z}}
\newcommand{\TT}{\mathbb{T}}

\newcommand{\CC}{\mathbb{C}}
\newcommand{\RR}{\mathbb{R}}
\newcommand{\HH}{\mathbb{H}}
\newcommand{\QQ}{\mathbb{Q}}
\newcommand{\EE}{\mathbb{E}}

\DeclareMathOperator{\lin}{span}
\DeclareMathOperator{\Irr}{Irr}
\DeclareMathOperator{\Pol}{Pol}

\DeclareMathOperator{\Tr}{Tr}
\DeclareMathOperator{\B}{B}
\DeclareMathOperator{\M}{M}
\DeclareMathOperator{\N}{N}

\DeclareMathOperator{\LL}{L}

\DeclareMathOperator{\HS}{HS}

\DeclareMathOperator{\sot}{sot}

\makeatletter
\@namedef{subjclassname@2020}{\textup{2020} Mathematics Subject Classification}
\makeatother

\newtheorem{proposition}{Proposition}[section]
\newtheorem{theorem}[proposition]{Theorem}
\newtheorem{theoremlet}{Theorem}

\newtheorem{lemma}[proposition]{Lemma}
\newtheorem{corollary}[proposition]{Corollary}

\theoremstyle{definition}
\newtheorem{definition}[proposition]{Definition}
\newtheorem{remark}[proposition]{Remark}

\numberwithin{equation}{section}
\begin{document}

\author{Jacek Krajczok}
\address{Institute of Mathematics of the Polish Academy of Sciences, Warsaw, Poland}
\email{jkrajczok@impan.pl}

\author{Mateusz Wasilewski}
\address{\parbox{\linewidth}{
Department of Mathematics, KU Leuven, Belgium \\
\emph{Currently}: Institute of Mathematics of the Polish Academy of Sciences, Warsaw, Poland}}
\email{mateusz.wasilewski@kuleuven.be}

\begin{abstract}
We study analogues of the radial subalgebras in free group factors (called the algebras of class functions) in the setting of compact quantum groups. For the free orthogonal quantum groups we show that they are not MASAs, as soon as we are in a non-Kac situation. The most important notion to our present work is that of a (quasi-)split inclusion. We prove that the inclusion of the algebra of class functions is quasi-split for some unitary quantum groups -- in this case the subalgebra is non-abelian and we also obtain a result concerning its relative commutant. In the positive direction, we construct certain bicrossed products from the quantum group $\SU_q(2)$ for which the algebra of class functions is a MASA. 
\end{abstract}

\title{On the von Neumann algebra of class functions on a compact quantum group}

\keywords{Compact quantum group, Quasi-split inclusion, Maximal abelian subalgebra}

\subjclass[2020]{Primary 46L67, 20G42}

\maketitle

\section{Introduction}
It is a well established fact that discrete (quantum or classical) groups provide a source of interesting examples to the theory of von Neumann algebras. One of the most prominent examples of von Neumann algebras associated with groups are the free group factors. Let $F_n$ the free group with generators $g_1,\dotsc, g_n\,(n\ge 2)$ and let $\LL(F_n)=\lambda(F_n)''$ be the corresponding group von Neumann algebra. Inside $\LL(F_n)$ one finds the so called radial subalgebra $\mscr{R}$, the von Neumann algebra generated by the operator $(\lambda_{g_1}+\lambda_{g_1}^*)+\cdots +(\lambda_{g_n}+\lambda_{g_n}^*)$. Its name stems from the property that if we (informally) write $x=\sum_{w\in F_n} f(w)\lambda_w\in \LL(F_n)$ then $x\in \mscr{R}$ if and only if $f$ is a radial function on $F_n$, i.e. $f(w)$ depends only on the length $|w|$ of $w$. The radial algebra was intensively studied -- let us mention that it was proved to be maximal abelian (MASA) in $\LL(F_n)$ \cite{Pytlik} and later to be singular \cite{Radulescu} and even maximal injective \cite{inj}.\\

One may look at the element $(\lambda_{g_1}+\lambda_{g_1}^*)+\cdots +(\lambda_{g_n}+\lambda_{g_n}^*)$ from a different perspective. It is the character of the fundamental representation of the compact quantum group $\widehat{F_n}$, dual to $F_n$. Consequently, the radial algebra $\mscr{R}$ is the von Neumann algebra generated by this character.
It is therefore natural to wonder whether similar properties hold for other discrete (or, by duality, compact) quantum groups. This question was studied in particular in the case of the (Kac type) free orthogonal quantum group $O_N^+$. In \cite{RadialMasa} it was shown that $\mscr{C}_{O_N^+}$, the von Neumann algebra generated by the character of the fundamental representation, is a singular MASA in $\LL^{\infty}(O_N^+)$. Observe that now $\mscr{C}_{O_N^+}$ has also a different description -- it coincides with the von Neumann algebra generated by all characters of irreducible representations. Let us take this description as a general definition of $\mscr{C}_{\GG}$ for a compact quantum group $\GG$\footnote{Note however that for $\GG=\widehat{F_n}$ we do not have equality of $\mscr{R}$ and $\mscr{C}_{\GG}$. In fact, $\mscr{C}_{\GG}=\LL^{\infty}(\GG)$ holds for all abelian compact quantum groups (meaning that the function algebra $\LL^{\infty}(\GG)$ is cocommutative).}.

\begin{definition}[{See \cite[after Corollary 5.10]{pseudogroups}}]\label{def:classfun}
For a compact quantum group $\GG$ we define the von Neumann algebra of class functions $\mscr{C}_{\GG}:=\{\chi_{\alpha}\,|\,\alpha\in\IrrG\}'' \subseteq L^{\infty}(\GG)$ .
\end{definition}

We choose to call $\mscr{C}_{\GG}$ ``the von Neumann algebra of class functions'' because the two coincide for classical compact groups (see Lemma \ref{lem:classfunctions}).

In this article we will focus on non-Kac type compact quantum groups, thus in many cases we enter the realm of type $\operatorname{III}$ von Neumann algebras. It means that we cannot use one of the most important tools -- the conditional expectation -- unless our subalgebra is preserved by the modular group. In fact, many examples of MASAs in type III factors admit a conditional expectation, for example the Cartan subalgebras inside crossed products by non-singular actions or the von Neumann subalgebra generated by the positive part in the polar decomposition of a generalized circular element in the free Araki-Woods factors (see \cite[Theorem 4.8]{MR1444786}). However, the abelian subalgebras of free Araki-Woods factors that are counterparts of the famous generator subalgebras of the free group factors, are not preserved by the modular group and are not MASAs (this is also known in the case of $q$-Araki-Woods algebras, see \cite{Bikram}). The crucial notion in this setting is that of a (quasi-)split inclusion (see \cite{MR345546} and \cite{MR735338}), which will also be central to our present work.

In Section \ref{sec:quasi-split} we will discuss the basics of quasi-split inclusions of von Neumann algebras. We record there precise statements that will allow us to prove that the inclusion of the von Neumann algebra of class functions is quasi-split in some cases. In our first result we obtain a concrete criterion in terms of the relation between the usual and quantum dimensions of irreducible representations (this criterion is satisfied for many, but not all, non-Kac type quantum groups, see Remark \ref{Rem:UFnonexample}).
\begin{theoremlet}[{Theorem \ref{Thm:quasi-split}}]
Let $\GG$ be a compact quantum group. Suppose that $\sum_{\alpha \in \Irr(\GG)} \sqrt{\tfrac{\op{dim}(\alpha)}{\op{dim}_q(\alpha)}}$ converges. Then the inclusion of the von Neumann algebra of class functions into $\LL^{\infty}(\GG)$ is quasi-split.
\end{theoremlet}

This result is sufficient to prove that the von Neumann algebra of class functions is not a MASA in the case that $\LL^{\infty}(\GG)$ is a type $\op{III}$ factor. One of the cases of most interest to us is the compact quantum group $O_{F}^{+}$, for which it is not known in full generality whether $\LL^{\infty}(O_{F}^{+})$ is a factor. Therefore in Section \ref{sec:3} we develop a new approach, involving the scaling group, a unique symmetry of the non-Kac type quantum groups. What we show, in broad strokes, is that if being non-Kac is witnessed by all non-trivial irreducible representations then the von Neumann algebra of class functions cannot be a MASA as soon as the inclusion is quasi-split. More precisely, we prove that if it were a MASA then $\LL^{\infty}(\GG)$ would have to be factor, and the quasi-split property would force this factor to be of type $\op{I}$. In this case a non-trivial inner scaling automorphism must exist, which we show to be excluded by our assumptions. 

\begin{theoremlet}[{Theorem \ref{Thm:nonmasa}}]
Let $\GG$ be a non-trivial compact quantum group such that the inclusion of the von Neumann algebra of class functions is quasi-split and $\uprho_{\alpha}  \neq \mathds{1}_{\alpha}$ for all non-trivial representations $\alpha \in \Irr(\GG)$. Then the von Neumann algebra of class functions is not a MASA.
\end{theoremlet}

Here $\uprho_{\alpha} \in \op{B}(\mathsf{H}_{\alpha})$ is the unique positive matrix intertwining the irreducible representation $\alpha$ with its double conjugate and satisfying $\op{Tr}(\uprho_{\alpha}) = \op{Tr}(\uprho_{\alpha}^{-1})$; for an arbitrary finite dimensional representation $U$ we can define $\uprho_{U}$ using the decomposition of $U$ into irreducible representations, see \cite[Section 1.4]{NeshTu}.

In Section \ref{sec:examples} we focus on studying concrete examples. More specifically, in Subsection \ref{subsec:abelian} we introduce a class of compact quantum groups for which we can verify the criterion from Section \ref{sec:quasi-split} for the quasi-split property. Combined with the results from Section \ref{sec:3} we obtain the following.
\begin{theoremlet}[{Corollary \ref{Cor:quasisplit}}]
Let $\GG = O_{F}^{+}$ -- the free orthogonal quantum group -- or $\GG = \GG_{Aut}(B,\psi)$ -- the quantum automorphism group of a finite dimensional $C^{\ast}$-algebra equipped with a $\delta$-form. If $\GG$ is not of Kac type then $\mscr{C}_{\GG}$ is not a MASA in $\LL^{\infty}(\GG)$.
\end{theoremlet}

In the same section we study the compact quantum group $\SU_q(2)$. In this case the von Neumann algebra of class functions is not a MASA in $\LL^{\infty}(\SU_q(2))$, but it is a MASA in the fixed point subalgebra of the scaling group. From the action by scaling automorphisms (restricted to the rational numbers, treated as a discrete group) we build a new compact quantum group using the bicrossed product construction. It is a non-Kac type compact quantum group for which we can prove that the von Neumann algebra of class functions is a MASA.

\begin{theoremlet}[{Proposition \ref{prop4.2.1}}]
Let $\HH$ be the bicrossed product $\QQ\bowtie\SU_q(2)$. Then $\mscr{C}_{\HH}$ is a MASA in $\LL^{\infty}(\HH)$.
\end{theoremlet}

The bicrossed product construction above actually depends on the choice of a non-zero real number and for most choices $\LL^{\infty}(\HH)$ is the unique injective factor of type $\op{II}_{\infty}$.

In the last section of the paper we investigate the properties of the free unitary groups $U_{F}^{+}$. In this case the von Neumann algebra of class functions is non-commutative, but it remains tracial (see \cite[Equation (5.33)]{pseudogroups}). We show that the inclusion is quasi-split as long as the quantum group is ``sufficiently non-Kac''.

\begin{theoremlet}[{See Corollary \ref{Cor:UF} for a precise statement}]
Let $\GG= U_{F}^{+}$. If the ratio $\frac{\dim(\alpha)}{\dim_q(\alpha)}$, where $\alpha$ is the fundamental representation, is sufficiently small then the inclusion $\mscr{C}_{U_F^+} \subseteq \LL^{\infty}(U_F^+)$ is quasi-split. As a consequence, the relative commutant $\mscr{C}_{U_F^+}' \cap \LL^{\infty}(U_F^+)$ is not contained in $\mscr{C}_{U_F^+}$.
\end{theoremlet}

\subsection{Preliminaries}
We would like to start by motivating our terminology from Definition \ref{def:classfun}.

\begin{lemma}\label{lem:classfunctions}
Let $G$ be a compact group. Then $\{\chi_\alpha\,|\, \alpha\in \Irr(G)\}''$ equals the von Neumann algebra of essentially bounded measurable class functions, i.e. the set of ~$f\in\LL^{\infty}(G)$ satisfying $f(hgh^{-1})=f(g)\,(g,h\in G)$.
\end{lemma}

\begin{proof}
Since every character $\chi_\alpha$ is a class function, one of the inclusions is clear. Let $\mu_G$ be the Haar measure on $G$ and $\EE \colon f\mapsto \int_G f(h\cdot h^{-1})\md\mu_G(h)\;(f\in\LL^{\infty}(G))$ the normal conditional expectation onto the von Neumann algebra of essentially bounded measurable class fuctions. As matrix coefficients of irreducible representations span a $w^*$-dense subspace in $\LL^{\infty}(G)$, it is enough to show that $\EE(u^{\alpha}_{\xi,\eta})= \frac{\ismaa{\xi}{\eta}}{\dim(\alpha)} \chi_{\alpha}$ for all $\alpha\in\Irr(G)$ and vectors $\xi,\eta\in \msf{H}_\alpha$. Using orthogonality relations one can check that $\ismaa{\Lambda_h( \chi_{\alpha})}{ \Lambda_h(u_{v,w}^{\beta})} = \delta_{\alpha \beta} \frac{\ismaa{v}{w}}{\dim(\alpha)}$, so it is sufficient to check that $\ismaa{\Lambda_h( \mathbb{E}(u_{\xi, \eta}^{\alpha}))}{\Lambda_h( u_{v,w}^{\beta})} = \frac{\delta_{\alpha \beta}}{(\dim(\alpha))^{2}}\ismaa{ \eta}{ \xi}\ismaa{v}{w} $. This can, once again, be verified using the orhogonality relations:
\begin{align*}
\ismaa{\Lambda_h( \mathbb{E}(u_{\xi, \eta}^{\alpha}))}{ \Lambda_h(u_{v,w}^{\beta})} &= \int_{G} \int_{G} \ismaa{ \alpha(hgh^{-1}) \eta}{ \xi}\ismaa{v}{\beta(g) w} \md\mu_{G}(h)\md\mu_{G}(g) \\
&= \int_{G} \ismaa{\Lambda_h( u_{\alpha(h^{-1}) \xi, \alpha(h^{-1})\eta}^{\alpha})}{\Lambda_h( u_{v,w}^{\beta})} \md\mu_{G}(h) \\
&=\frac{ \delta_{\alpha \beta}}{\dim(\alpha)} \int_{G} \ismaa{v}{ \alpha(h^{-1})\xi}\ismaa{ \alpha(h^{-1}) \eta}{w} \md\mu_{G}(h) \\
&= \frac{\delta_{\alpha \beta}}{\dim(\alpha)} \ismaa{\Lambda_h( u_{\xi, v}^{\alpha})}{\Lambda_h( u_{\eta, w}^{\alpha})} \\
&= \frac{\delta_{\alpha \beta}}{(\dim(\alpha))^2} \ismaa{ \eta}{ \xi}\ismaa{ v}{w}.
\end{align*}
\end{proof}
\begin{remark}
Actually, it has been proved in \cite[Theorem 3.7]{MR3679720} that for any compact quantum group $\GG$ the algebra of class functions $\mscr{C}_{\GG}$ is equal to $\{x \in \LL^{\infty}(\GG): \Delta(x) = \Sigma \Delta(x)\}$, where $\Delta: \LL^{\infty}(\GG) \to \LL^{\infty}(\GG) \overline{\otimes} \LL^{\infty}(\GG)$ is the comultiplication and $\Sigma: \LL^{\infty}(\GG) \overline{\otimes} \LL^{\infty}(\GG) \to \LL^{\infty}(\GG) \overline{\otimes} \LL^{\infty}(\GG)$ is the flip; this answered a question left open by Woronowicz, see \cite[after Proposition 5.11]{pseudogroups}. In the classical case it corresponds to functions $f$ satisfying $f(gh) = f(hg)$, i.e. precisely the class functions.
\end{remark}
\begin{lemma}[{\cite[Equation (5.33)]{pseudogroups}}]
Let $\GG$ be a compact quantum group. The Haar integral integral $h$ restricted to $\mscr{C}_{\GG}$ is tracial.
\end{lemma}
In most of the examples $\mscr{C}_{\GG}$ will actually be commutative (the fusion rules will be commutative), with the important exception of the free unitary quantum groups $U_{F}^{+}$. 

We will work with non-Kac type compact quantum groups and these objects naturally admit two one-parameter automorphism groups: the modular group (because the Haar integral is non-tracial) and the scaling group (because the antipode is not involutive). For more information on compact quantum groups consult e.g.~\cite{NeshTu}.

There will be two main classes of compact quantum groups to which our results apply: we will start with free orthogonal (unitary) quantum groups.
\begin{definition}[{\cite{MR1382726}}]
Let $N\geqslant 2$ be an integer and let $F$ be an invertible matrix in $M_N(\CC)$. Let
\begin{enumerate}[(i)]
\item $\Pol(U_F^{+})$ be the universal $\ast$-algebra generated by the entries of a unitary matrix $U \in M_{N}(\Pol(U_F^{+}))$ subject to the condition that $F\overline{U} F^{-1}$ is also a unitary matrix;
\item $\Pol(O_{F}^{+})$ be the universal $\ast$-algebra generated by the entries of a unitary matrix $U \in M_{N}(\Pol(O_F^{+}))$ subject to the condition that $U=F\overline{U} F^{-1}$; we also assume in this case that $F\overline{F} = c\mathds{1}$ for some $c\in \RR$.
\end{enumerate}
Then the comultiplication $\Delta(u_{ij}):= \sum_{k=1}^{N} u_{ik} \otimes u_{kj}$ gives both $\Pol(U_F^{+})$ and $\Pol(O_F^{+})$ the structure of a Hopf $\star$-algebra. After the universal $\mathrm{C}^*$-completion, they give rise to compact quantum groups $U_{F}^{+}$ and $O_{F}^{+}$.
\end{definition}

Before introducing the quantum automorphism groups of finite-dimensional $\mathrm{C}^{\ast}$-algebras, we need to recall the definition of a $\delta$-form. Let $B$ be a finite-dimensional $C^{\ast}$-algebra, let $m: B \otimes B \to B$ be the multiplication map, and let $\psi: B \to \CC$ be a faithful state. Then $\psi$ defines natural inner products on both $B$ and $B\otimes B$ and we can compute the Hermitian adjoint of $m$, i.e. $m^{\ast}$. We say that $\psi$ is a $\delta$-form if $mm^{\ast} = \delta^2 \op{Id}$.

\begin{definition}[{\cite{MR1637425}}]
Let $B$ be a finite-dimensional $C^{\ast}$-algebra equipped with a $\delta$-form $\psi$. Then there exists a universal compact quantum group acting on $B$ in a $\psi$-preserving way -- we denote it by $\GG_{Aut}(B,\psi)$ and the corresponding action by $\alpha$. 

It means that whenever we have a compact quantum group $\GG$ with a $\psi$-preserving action $\beta: B \to \mathrm{C}_u(\GG) \otimes B$, i.e. $\beta$ is a $\ast$-homomorphism satsifying $(\Delta \otimes \op{Id})\circ \beta = (\op{Id} \otimes \beta)\circ \beta$ and $(\op{Id} \otimes \psi) (\beta(f)) = \psi(f) \mathds{1}$, then there exists a unique quantum group morphism $\Phi: \mathrm{C}_u(\GG_{Aut}(B,\psi)) \to \mathrm{C}_u(\GG)$ such that $(\Phi\otimes \op{Id})\alpha = \beta$.\\
Let us mention here that in the Kac case, operator algebras associated with $\GG_{Aut}(B,\psi)$ were extensively studied e.g.~in \cite{BrannanTrace}.
\end{definition}

\subsection{Notation}
For a faithful normal state $\omega$ on a von Neumann algebra $\N$ we will denote by $(\sigma^{\omega}_t)_{t\in\RR}$ the modular group, by $\nabla_\omega$ the modular operator, by $J_\omega$ the modular conjugation and by $\Lambda_{\omega}$ the canonical map $\N\rightarrow \LL^2(\N,\omega)$. When there is no risk of confusion, we will simply write $(\sigma_t)_{t\in\RR}$ etc. Similarly, for a compact quantum group  $\GG$ we will often write $(\tau_t)_{t\in\RR}$ for the scaling group. We will use von Neumann subalgebras $\LL^{\infty}(\GG)^{\tau},\LL^{\infty}(\GG)^{\sigma}\subseteq\Linf$ of those $x\in \Linf$ which are invariant under the scaling group (respectively~the modular group of the Haar integral $h$). Algebraic tensor products will be denoted by $\odot$.

\section{Quasi-split inclusions}\label{sec:quasi-split}
The split property arose in the study of inclusions of von Neumann algebras appearing in algebraic quantum field theory (see \cite{MR345546}). A more systematic study of related properties was conducted in \cite{MR735338}. The main tool for proving that a given inclusion is split was developed in \cite{MR1038440}, albeit only in the case of factor inclusions, and we need to resort to subsequent work \cite{MR1855628}, which deals with the general case. The precise criterion that we will use comes from \cite{Bikram}. 
\begin{definition}\label{def:quasi-split}
Let $\msf{N}\subseteq \msf{M}$ be an inclusion of von Neumann algebras. We say this inclusion is \textbf{split} if there exists an intermediate type $\operatorname{I}$ factor $\msf{B}$. We say that the inclusion is \textbf{quasi-split} if the $\ast$-homomorphism $\msf{N} \odot \msf{M}^{op} \ni x\otimes y^{op} \mapsto x Jy^{\ast} J \in \B(\LL^{2}(\mathsf{M}))$, where $L^{2}(\mathsf{M})$ is the standard form, extends to a surjective normal $\ast$-homomorphism from $\msf{N} \overline{\otimes} \msf{M}^{op}$ onto $\msf{N} \vee \msf{M}^{\prime}$.
\end{definition}
\begin{remark}\label{rem:splitfactor}
If both $\msf{N}$ and $\msf{M}$ are factors or one of them is a type $\operatorname{III}$ algebra then a quasi-split inclusion is automatically split (see \cite[Corollary 1, (iv) and (vi)]{MR703083}).
\end{remark}
The most important consequence of the quasi-split property for our work is that it can be used as a tool for proving that some abelian subalgebras are not MASAs.
\begin{proposition}[{\cite[Corollary 3.11]{Bikram}}]\label{Prop:notmasa}
If $\mathsf{M}$ is a type $\operatorname{III}$ von Neumann algebra and $\msf{N} \subseteq \msf{M}$ is a quasi-split inclusion then $\msf{N}^{\prime} \cap \msf{M}$ is also of type $\operatorname{III}$. In particular $\msf{N}$ cannot be a MASA.
\end{proposition}
We would like to present now a useful criterion for proving that a given inclusion is quasi-split, which is a variant of Proposition 2.3 from \cite{MR1038440}. Recall first that if a von Neumann algebra $\msf{M}$ is represented on a Hilbert space $\msf{H}$ with a cyclic and separating vector $\Omega$ then $\msf{H}$ can be identified with $\LL^{2}(\msf{M})$ and we have an inclusion $\Phi_{2}: \msf{M} \to \LL^{2}(\msf{M})$ given by $x \mapsto \nabla^{\frac{1}{4}} x\Omega$. 

We will also need the notion of a nuclear map between two Banach spaces $X$ and $Y$. A map $T:X \to Y$ is called nuclear if there are sequences $(y_n)_{n\in \NN} \subseteq Y$ and $(x_n^{\ast})_{n\in \NN} \subseteq X^{\ast}$ such that $Tx = \sum_{n\in \NN} x_{n}^{\ast}(x) y_n$ and $\sum_{n\in \NN} \|x_n^{\ast}\|  \|y_n\| < \infty$.

\begin{proposition}[{\cite[Proposition 3.7]{Bikram}}]\label{Prop:nuclear}
Let $\mathsf{N} \subseteq \mathsf{M}$ be a pair of von Neumann algebras. If the map $\Phi_{2| \mathsf{N}}: \msf{N} \to \LL^{2}(\msf{M})$ is nuclear then the inclusion $\mathsf{N} \subseteq \msf{M}$ is quasi-split.
\end{proposition}

\subsection{Quasi-split inclusion of the algebra of class functions}
In this subsection we will present examples of compact quantum groups for which the inclusion of the algebra of class functions is quasi-split. Before doing that we need some preparation regarding the action of the modular group on characters of unitary representations.

Recall that our aim is to show that the map $\Phi_{2|\mscr{C}_{\GG}}: \mscr{C}_{\GG} \to \LL^{2}(\GG)$ given by $x \mapsto \nabla^{\frac{1}{4}} \Lambda_{h}(x)$ is nuclear. Note that $\mscr{C}_{\GG}$ is given by the closed linear span of the characters of the (finite dimensional) unitary representations and these characters are analytic elements for the modular group. Therefore $\nabla^{\frac{1}{4}} \Lambda_h(\chi) = \Lambda_h(\sigma_{-\frac{i}{4}}(\chi))$ holds for every character $\chi$; we first have to understand the action of the modular group on the characters. 

Recall that (see \cite[Page 30]{NeshTu}) for any representation $U$ on $\msf{H}_{U}$ we have $(\id \otimes \sigma_z)(U) = (\uprho_U^{iz} \otimes \I_{U})U(\uprho_U^{iz} \otimes \I_{U})$. If we choose an orthonormal basis of $\msf{H}_{U}$ in which $\uprho_{U}$ is diagonal then we can write more concretely that $\sigma_{z}(u_{kl}) = \uprho_{U,k}^{iz} \uprho_{U,l}^{iz} u_{kl}$. Therefore for the character $\chi_U:= \sum_{k} u_{kk}$ we have $\sigma_z(\chi_U) = \sum_{k} \uprho_{U,k}^{2iz} u_{kk}$. We will now compute the $\LL^{2}$-norm of this element.
\begin{lemma}\label{modularcharacter}
If $U$ is irreducible then we have $\|\sigma_{a+ib}(\chi_U)\|_2^2 = \frac{\op{Tr} (\uprho_U^{-4b-1})}{\op{Tr}(\uprho_U)}$ for all $a,b\in\RR$.
\end{lemma}
\begin{proof}
Recall that by definition $\|x\|_2^2 = h(x^{\ast}x)$, so in our case we get
\[
\|\sigma_{a+ib}(\chi_U)\|_2^2 = \sum_{k,l} \uprho_{U,k}^{-2ia-2b} \uprho_{U,l}^{2ia - 2b} h(u_{kk}^{\ast} u_{ll}).
\]
Using the orthogonality relations (see \cite[Theorem 1.4.3]{NeshTu}) we get $h(u_{kk}^{\ast} u_{ll}) = \delta_{k,l} \frac{\uprho^{-1}_{U,k}}{\op{dim}_q(U)}$. Therefore we obtain
\[
\|\sigma_{a+ib}(\chi_U)\|_2^2 = \sum_{k} \uprho_{U,k}^{-4b} \frac{\uprho_{U,k}^{-1}}{\op{dim}_q(U)} = \frac{\op{Tr}(\uprho_{U}^{-4b-1})}{\op{dim}_q(U)}.
\]
To finish the proof we just have to recall that $\dim_q(U) = \op{Tr}(\uprho_U)$.
\end{proof}
\begin{corollary}\label{Cor:modularchar}
We have $\|\chi_U\|_2=1$ and $\|\sigma_{-\frac{i}{4}}(\chi_U)\|_2^2 = \frac{\op{dim}(U)}{\op{dim}_q(U)}$.
\end{corollary}
The relation between quantum dimension and the usual dimension will be crucial for proving that the inclusion of the algebra of class functions is quasi-split. We will present here a precise criterion and then describe a class of compact quantum groups to which it applies.
\begin{theorem}\label{Thm:quasi-split}
Let $\GG$ be a compact quantum group. Suppose that $\sum_{\alpha \in \Irr(\GG)} \left(\frac{\op{dim}(\alpha)}{\op{dim}_q(\alpha)}\right)^{\frac{1}{2}} < \infty$. Then the inclusion $\mscr{C}_{\GG} \subseteq \LL^{\infty}(\GG)$ is quasi-split.
\end{theorem}

\begin{proof}
We want to show that the map $\Phi_{2|\mscr{C}_{\GG}}: \mscr{C}_{\GG} \to \LL^{2}(\GG)$ is nuclear. Recall that it is a composition of two maps: the inclusion $\Lambda_h: \mscr{C}_{\GG} \to \Lambda_h(\mscr{C}_{\GG})$ and $\nabla^{\frac{1}{4}}: \Lambda_h(\mscr{C}_{\GG}) \to \LL^{2}(\GG)$. 

We will first show that $\nabla^{\frac{1}{4}}$ extends to a contraction from $\LL^{2}(\mscr{C}_{\GG}):= \overline{\Lambda_h(\mscr{C}_{\GG})}$ to $\LL^{2}(\GG)$. Note that the characters of irreducible representations are linearly dense in $\LL^{2}(\mscr{C}_{\GG})$ and the characters form an orthonormal set, by the orthogonality relations. Let $x:= \sum_{\alpha} c_{\alpha}\chi_{\alpha}$ be a finite sum of characters of irreducible representations. Note that $\|\Lambda_h(x^{\ast})\|^2 = \|\sum_{\alpha} \overline{c_{\alpha}} \Lambda_h(\chi_{\overline{\alpha}})\|^2 = \sum_{\alpha} |c_{\alpha}|^2 = \|\Lambda_h(x)\|^2$. It follows that 
\begin{align*}
\|\nabla^{\frac{1}{4}} \Lambda_h(x)\|^2 &= \ismaa{\Lambda_h(x)}{\nabla^{\frac{1}{2}} \Lambda_h(x)} \\
&= \ismaa{J\nabla^{\frac{1}{2}} \Lambda_h(x)}{J \Lambda_h(x)} \\
&= \ismaa{\Lambda_h(x^{\ast})}{J \Lambda_h(x)} \\
&\leqslant \|\Lambda_h(x^{\ast})\|\cdot \|\Lambda_h(x)\| = \|\Lambda_h(x)\|^2,
\end{align*}
so $\nabla^{\frac{1}{4}}$ extends to a contraction from $\LL^{2}(\mscr{C}_{\GG})$ to $\LL^2(\GG)$.

We will now show that $\Phi_2$ from $\mscr{C}_{\GG}$ to $\LL^{2}(\GG)$ is a nuclear map. We have $\Phi_2(x) = \nabla^{\frac{1}{4}} \Lambda_h(x)$. As $\Lambda_h(x) \in \LL^{2}(\mscr{C}_{\GG})$, we can write $\Lambda_h(x) = \sum_{\alpha} \ismaa{ \Lambda_{h}(\chi_{\alpha})}{ \Lambda_{h}(x)} \Lambda_h(\chi_{\alpha})$. As $\nabla^{\frac{1}{4}}: \LL^{2}(\mscr{C}_{\GG}) \to \LL^{2}(\GG)$ is bounded, we have
\[
\Phi_2(x) = \sum_{\alpha \in \IrrG} \ismaa{ \Lambda_{h}(\chi_{\alpha})}{ \Lambda_h(x)} \nabla^{\frac{1}{4}}\Lambda_h(\chi_{\alpha}).
\]
If we define functionals $\omega_{\alpha}(x):= \ismaa{\Lambda_{h}(\chi_{\alpha})}{\Lambda_{h}(x)}$ then it suffices to check that $\sum_{\alpha} \|\omega_{\alpha}\|\cdot \|\nabla^{\frac{1}{4}} \Lambda_h(\chi_{\alpha})\| < \infty$. We already know that $\|\nabla^{\frac{1}{4}} \Lambda_h(\chi_{\alpha})\| = \left(\frac{\op{dim}(\alpha)}{\op{dim}_q(\alpha)}\right)^{\frac{1}{2}}$ and it is clear that $\|\omega_{\alpha}\|\leqslant \|\chi_{\alpha}\|_2 = 1$, hence
\[
\sum_{\alpha\in\IrrG} \|\omega_{\alpha}\|\cdot \|\nabla^{\frac{1}{4}} \Lambda_h(\chi_{\alpha})\| \leqslant \sum_{\alpha\in\IrrG} \left(\frac{\op{dim}(\alpha)}{\op{dim}_q(\alpha)}\right)^{\frac{1}{2}} < \infty.
\] 
By Proposition \ref{Prop:nuclear} the inclusion $\mscr{C}_{\GG} \subseteq \LL^{\infty}(\GG)$ is quasi-split.
\end{proof}
This result is already enough to prove that in many cases the radial subalgebra in $\LL^{\infty}(O_{F}^{+})$ is not a MASA; it follows from \cite{Boundary} that $\LL^{\infty}(O_{F}^{+})$ is often a type $\operatorname{III}$ factor and we can use Proposition \ref{Prop:notmasa}. We will generalize this result in the next section (see Corollary \ref{Cor:quasisplit}).

\section{Relative commutant of $\mscr{C}_{\GG}$ and inner scaling automorphisms}\label{sec:3}

We will be interested in the relative commutant $\mscr{C}_{\GG}^{\prime} \cap \LL^{\infty}(\GG)$. If $\mscr{C}_{\GG}$ is commutative, as is the case for example for the free orthogonal quantum groups, then the condition $\mscr{C}_{\GG}^{\prime} \cap \LL^{\infty}(\GG) \subseteq \mscr{C}_{\GG}$ precisely means that $\mscr{C}_{\GG}$ is a MASA in $\LL^{\infty}(\GG)$.\\
Our strategy for proving that $\mscr{C}_{\GG}$ cannot be a MASA in many cases will be the following. We will show that if $\mscr{C}_{\GG}$ was a MASA then $\LL^{\infty}(\GG)$ would have to be a factor. Moreover, if the inclusion $\mscr{C}_{\GG}\subseteq \LL^{\infty}(\GG)$ is quasi-split then it would have to be a type $\operatorname{I}$ factor. We will then use properties of the scaling automorphisms to exclude this case. Let us now move on to more precise statements.
\begin{proposition}\label{Prop:masafactor}
Let $\GG$ be a compact quantum group. Then 
\[
 \overline{\op{span}}^{w^{\ast}}\{\chi_{\alpha}\,|\, \uprho_{\alpha} = \mathds{1}_{\alpha}\} = \mscr{C}_{\GG} \cap \LL^{\infty}(\GG)^{\sigma} = \bigcap_{t\in \mathbb{R}} \sigma_t(\mscr{C}_{\GG}).
\]
If moreover  $\mscr{C}_{\GG}^{\prime}\cap \LL^{\infty}(\GG) \subseteq \mscr{C}_{\GG}$ then
\[
\mathcal{Z}(\LL^{\infty}(\GG)) \subseteq  \overline{\op{span}}^{w^{\ast}}\{\chi_{\alpha}\,|\, \uprho_{\alpha} = \mathds{1}_{\alpha}\}
\]
In particular, if $\uprho_{\alpha} \neq \mathds{1}_{\alpha}$ for all non-trivial irreducible representations of $\GG$ then $\LL^{\infty}(\GG)$ is a factor.
\end{proposition}
\begin{proof}
It is clear that $\overline{\op{span}}^{w^{\ast}}\{\chi_{\alpha} \,|\, \uprho_{\alpha} = \mathds{1}_{\alpha}\} \subseteq \mscr{C}_{\GG} \cap \LL^{\infty}(\GG)^{\sigma} \subseteq \bigcap_{t\in\mathbb{R}} \sigma_t(\mscr{C}_{\GG})$. Now let $x\in \bigcap_{t\in\mathbb{R}} \sigma_t(\mscr{C}_{\GG})$; we can write $\Lambda_h(\sigma_t(x)) = \sum_{\alpha} c_{\alpha} \Lambda_h(\sigma_t(\chi_{\alpha})) = \sum_{\alpha} c_{\alpha, t} \Lambda_h(\chi_{\alpha})$. It follows from the orthogonality relations that the elements $\sigma_t(\chi_{\alpha})$ and $\chi_{\beta}$ are orthogonal unless $\alpha=\beta$. Therefore we get $c_{\alpha}\sigma_t(\chi_{\alpha}) = c_{\alpha,t} \chi_{\alpha}$ for all $t$, hence either $c_{\alpha} = c_{\alpha,t} = 0$ or $\chi_{\alpha}$ is an eigenvector of the modular group. If $\alpha$ is an irreducible representation on the Hilbert space $\mathsf{H}_\alpha$ and we pick an orthonormal basis $(e_k)_k$ in which the matrix $\uprho_{\alpha}$ is diagonal then $\sigma_z(\chi_{\alpha}) = \sum_{k} \uprho_{\alpha,k}^{2iz} u^{\alpha}_{k,k}$, where $\chi_{\alpha} = \sum_{k} u^{\alpha}_{k,k}$. Thus $\chi_{\alpha}$ is an eigenvector of the modular group iff the matrix $\uprho_{\alpha}$ is a multiple of the identity, which actually forces $\uprho_{\alpha} = \mathds{1}$ due to the condition $\Tr(\uprho_{\alpha}) = \Tr(\uprho_{\alpha}^{-1})$. It follows that if a character of an irreducible representation is an eigenvector of the modular group then it is in fact a fixed point. 

Denote $A:= \overline{\op{span}}^{w^{\ast}}\{\chi_{\alpha} \,|\, \uprho_{\alpha} = \mathds{1}_{\alpha}\}$. We showed that any element $x\in \bigcap_{t\in\mathbb{R}} \sigma_t(\mscr{C}_{\GG})$  satisfies $\Lambda_h(x) \in \overline{\Lambda_h(A)}$. Since $A$ is contained in the centralizer, there exists a normal, faithful, state-preserving conditional expectation onto it, and it is easy to conclude that it implies $x\in A$. Alternatively one can invoke Lemma \ref{lemma4.2.2} because the Haar integral on the algebra $A$ is tracial. 

To finish the proof, note that the condition $\mscr{C}_{\GG}^{\prime} \cap \LL^{\infty}(\GG)\subseteq \mscr{C}_{\GG}$ implies that $\mathcal{Z}(\LL^{\infty}(\GG)) \subseteq \mscr{C}_{\GG}$. Moreover the center is always contained in the centralizer, so we obtain $\mathcal{Z}(\LL^{\infty}(\GG)) \subseteq \mscr{C}_{\GG} \cap \LL^{\infty}(\GG)^{\sigma}$.
\end{proof}
The next technical ingredient, featuring the scaling group, is the following proposition.
\begin{proposition}\label{Prop:scaling}
Let $\GG$ be a compact quantum group and let $t\in \RR$. Suppose that $\mscr{C}_{\GG}^{\prime} \cap \LL^{\infty}(\GG) \subseteq \mscr{C}_{\GG}$, the scaling automorphism $\tau_t$ is inner and is implemented by $v \in \LL^{\infty}(\GG)$. Then $v\in \overline{\op{span}}^{w^{\ast}}\{\chi_{\alpha}\,|\, \uprho_{\alpha} = \mathds{1}_{\alpha}\}$.
\end{proposition}
\begin{proof}
The scaling group acts trivially on characters, so we have $\chi_{\alpha} = \tau_t(\chi_{\alpha}) = v\chi_{\alpha} v^{\ast}$ for any $\alpha \in \Irr(\GG)$. Therefore $v \in \mscr{C}_{\GG}^{\prime} \cap \LL^{\infty}(\GG) \subseteq \mscr{C}_{\GG}$. Because of that we can write $\Lambda_h(v) = \sum_{\alpha} c_{\alpha} \Lambda_h(\chi_{\alpha})$.

We will now use the fact that the scaling group and the modular group commute, so for any $x\in \LL^{\infty}(\GG)$ we can write
\[
v x v^{\ast} = \tau_t(x) = \sigma_s \tau_t(\sigma_{-s}(x)) = \sigma_s(v \sigma_{-s}(x) v^{\ast}) = \sigma_s(v) x \sigma_{s}(v^{\ast}).
\]
It follows that $\sigma_s(v^{\ast})v x = x \sigma_s(v^{\ast}) v$, hence $\sigma_s(v^{\ast}) v \in \mathcal{Z}(\LL^{\infty}(\GG))$ for every $s\in \mathbb{R}$. We can write $\sigma_s(v) = v w_s$ with $w_s = v^{\ast} \sigma_s(v) \in \mathcal{Z}(\LL^{\infty}(\GG))$. As $\mathcal{Z}(\LL^{\infty}(\GG)) \subseteq \mscr{C}_{\GG}$, we have $vw_s \in \mscr{C}_{\GG}$. It follows that $\sigma_s(v) \in \mscr{C}_{\GG}$ for any $s\in \mathbb{R}$, i.e. $v \in \bigcap_{t\in\mathbb{R}} \sigma_t(\mscr{C}_{\GG})$; Proposition \ref{Prop:masafactor} allows us to conclude.
\end{proof}
Before we state and prove the main result of this section, we need to record another useful consequence of having a quasi-split inclusion.
\begin{lemma}[{\cite[Remark 3.10 (2)]{Bikram}}]\label{Lem:typeI}
Suppose that the inclusion $\msf{B} \subseteq \msf{M}$ is quasi-split and $\msf{B}$ is a MASA in $\msf{M}$. Then $\msf{M}$ is a direct sum of type $\operatorname{I}$ factors. 
\end{lemma}
\begin{theorem}\label{Thm:nonmasa}
Suppose that $\GG$ is a non-trivial compact quantum group such that the inclusion $\mscr{C}_{\GG}\subseteq \LL^{\infty}(\GG)$ is quasi-split and $\uprho_{\alpha} \neq \mathds{1}_{\alpha}$ for any non-trivial $\alpha \in \Irr(\GG)$. Then $\mscr{C}_{\GG}$ is not a MASA in $\LL^{\infty}(\GG)$.
\end{theorem}
\begin{proof}
Suppose that $\mscr{C}_{\GG}$ is a MASA in $\LL^{\infty}(\GG)$. It follows from Lemma \ref{Lem:typeI} that $\LL^{\infty}(\GG)$ is a direct sum of type $\operatorname{I}$ factors. Moreover, it follows from Proposition \ref{Prop:masafactor} that $\LL^{\infty}(\GG)$ is a factor, so $\LL^{\infty}(\GG)\simeq \op{B}(\msf{H})$ for some Hilbert space $\msf{H}$. 

By our assumptions there exists an irreducible representation $\alpha$ with $\uprho_{\alpha}\neq \mathds{1}_{\alpha}$, thus $\GG$ is not of Kac type, so there exists a non-trivial scaling automorphism $\tau_t$. All automorphisms of $\op{B}(\msf{H})$ are inner, so there exists $v\in \LL^{\infty}(\GG)$ implementing it. By Proposition \ref{Prop:scaling} $v \in \overline{\op{span}}^{w^{\ast}}\{\chi_{\alpha}\,|\, \uprho_{\alpha} = \mathds{1}_{\alpha}\} = \CC \mathds{1}$. But that means that $\tau_t$ is a trivial automorphism and this leads us to the desired contradiction.
\end{proof}
We will provide examples to which this result can be applied in Subsection \ref{subsec:abelian} (see Corollary \ref{Cor:quasisplit}).

\begin{remark}
In the proof of Theorem \ref{Thm:nonmasa} we argued using Proposition \ref{Prop:scaling} that if $\mscr{C}_{\GG}$ is a MASA in $\LL^{\infty}(\GG)$, then $\LL^{\infty}(\GG)$ is not isomorphic to $\B(\msf{H})$. In fact, one can show that $\LL^{\infty}(\GG)$ is not isomorphic to $\B(\msf{H})$ for any non-trivial compact quantum group $\GG$, see \cite{KrajczokSoltan}.
\end{remark}
\section{Examples}\label{sec:examples}

\subsection{Properties of $\SU_q(2)$}

Fix $q\in (-1,1)\setminus\{0\}$. In this subsection we establish a number of properties of $\mscr{C}_{\SU_q(2)}$: we show that it is not a MASA in $\LL^{\infty}(\SU_q(2))$ (Proposition \ref{prop4.1.1}), however it is a MASA in $\LL^{\infty}(\SU_q(2))^{\tau}$, the fixed point subalgebra of the scaling group (Proposition \ref{prop4.1.2}). This property will be used in the next subsection, where we construct a new compact quantum group out of $\SU_q(2)$ and $\QQ$, using a bicrossed product construction.\\

Let us start with recalling a description of $\SU_q(2)$ from the dual perspective. Since $\mathrm{C}(\SU_q(2))$ is a type $\operatorname{I}$, separable $\cst$-algebra \cite[Theorem A.2.3]{Woronowiczsu2}, $\widehat{\SU_q(2)}$ is a type $\operatorname{I}$ (second countable) discrete quantum group and we can use Desmedt's theorem to describe its structure using decompositions into direct integrals (see \cite{Desmedt} and \cite[Example 7.1]{KrajczokModular}). For each $\lambda\in \TT$ there is an irreducible representation
\[
\psi_{\lambda}\colon \mathrm{C}(\SU_q(2))\rightarrow \B(\msf{H}_\lambda)=\B(\ell^2(\ZZ_+)),
\]
given by
\[
\psi_\lambda(\alpha)\phi_k=\sqrt{1-q^{2k}}\phi_{k-1},\quad
\psi_{\lambda} (\gamma)\phi_k = \lambda q^k \phi_k\quad(k\in\ZZ_+)
\]
($\{\phi_{k}\}_{k\in\ZZ_+}$ is the standard orthonormal basis in $\ell^2(\ZZ_+),\, \phi_{-1}=0$ and $\alpha,\gamma$ are the standard generators of $\mathrm{C}(\SU_q(2))$). Denote by $\mu$ the normalized Lebesgue measure on the unit circle $\TT$.
Desmedt's theorem provides us with a unitary operator $\mc{Q}_L\colon \LL^2(\SU_q(2))\rightarrow \int_{\TT}^{\oplus} \HS(\msf{H}_\lambda)\md\mu(\lambda)$ given by
\[
\mc{Q}_L\colon \Lambda_h(a)\mapsto \int_{\TT}^{\oplus} \psi_\lambda(a) D_\lambda^{-1}\md\mu(\lambda)\quad(a\in \mathrm{C}(\SU_q(2)))
\]
where $D_\lambda^{-1}$ is a positive Hilbert-Schmidt operator
\[
D_\lambda^{-1}= \sqrt{1-q^2} \operatorname{diag}(1, |q|,|q|^{2},\dotsc)
\]
written with respect to the basis $\{\phi_k\}_{k\in\ZZ_+}$. Furthermore, we have
\begin{equation}\label{eq4.1.1}
\mc{Q}_L \LL^{\infty}(\SU_q(2)) \mc{Q}_L^*=
\int_{\TT}^{\oplus} \B(\msf{H}_\lambda)\otimes \I_{\ov{\msf{H}_\lambda}}
\,\md\mu(\lambda).
\end{equation}

The next lemma says that the von Neumann algebra generated by the real part of a weighted shift is MASA in $\B(\ell^2(\ZZ_+))$.

\begin{lemma}\label{lemma4.1.1}
Let $s\in \B(\ell^2(\ZZ_+))$ be the shift operator given by $s\phi_k=\phi_{k-1}\,(k\in\ZZ_+)$, and let $\M_f\in \B(\ell^2(\ZZ_+))$ be the multiplication operator associated with a function $f\in \ell^{\infty}(\ZZ_+)$. If $f(\NN)\subseteq \RR_{>0}$ then the von Neumann algebra $\mscr{B}$ generated by $T=s\M_f+\M_f s^*$ is maximal abelian in $\B(\ell^2(\ZZ_+))$.
\end{lemma}

\begin{proof}
According to \cite[Theorem 4.7.7]{AnalysisNow}, the claim follows once we show that there exists a cyclic vector for $\mscr{B}=\{T\}''$. We claim that $\phi_0$ is such a vector. Indeed, it is clear that $\phi_0$ belongs to $V=\ov{ \mscr{B} \phi_0}$. Next, assume that $\phi_0,\dotsc,\phi_n\in V$ for some $n\in\ZZ_+$. Then $\phi_{n+1}=\tfrac{1}{f(n+1)}T(\phi_n)-\tfrac{f(n)}{f(n+1)}\phi_{n-1}$ belongs to $V$ and consequently $V=\B(\ell^2(\ZZ_+))$.
\end{proof}

One easily sees that $\mscr{C}_{\SU_q(2)}$ is abelian and $\mscr{C}_{\SU_q(2)}=\{\alpha+\alpha^*\}''$. Indeed, $\alpha+\alpha^*$ is the character of the fundamental representation and the fusion rules of $\SU_q(2)$ imply that $\chi(U^n)\in\{\alpha+\alpha^*\}''$ for all $U^n\in\Irr(\SU_q(2))$. Consequently\footnote{By an abuse of notation, we will write $\psi_1(\mscr{C}_{\SU_q(2)})$ for the von Neumann algebra generated by $\psi_1(\alpha+\alpha^*)$.},
\begin{equation}\label{eq4.1.2}
\mc{Q}_L \mscr{C}_{\SU_q(2)}\mc{Q}_L^*=
\bigl\{\int_{\TT}^{\oplus} T\otimes\I_{\ov{\msf{H}_\lambda}}
\md\mu(\lambda)\,|\, T\in \psi_1(\mscr{C}_{\SU_q(2)})\bigr\}\simeq
\psi_1(\mscr{C}_{\SU_q(2)})\otimes \I_{\LL^{\infty}(\TT)}
\end{equation}
(observe that $\psi_\lambda(\alpha+\alpha^*)=\psi_1(\alpha+\alpha^*)$ for all $\lambda\in\TT$).
Note, using equation \eqref{eq4.1.2}, that $\mscr{C}_{\SU_q(2)}$ does not contain the center of $\LL^{\infty}(\SU_q(2))\simeq \B(\ell^2(\ZZ_{+}))\ov{\otimes} \LL^{\infty}(\TT)$, so it is certainly not a MASA. The next result describes its relative commutant.

\begin{proposition}\label{prop4.1.1}
The relative commutant of $\mscr{C}_{\SU_q(2)}$ is given by
\[\begin{split}
\mc{Q}_L (\mscr{C}_{\SU_q(2)}'\cap \LL^{\infty}(\SU_q(2)))\mc{Q}_L^*&=
\bigl\{\int_{\TT}^{\oplus} T_\lambda\otimes\I_{\ov{\msf{H}_\lambda}}
\md\mu(\lambda)\,|\,  T_\lambda \in \psi_1(\mscr{C}_{\SU_q(2)})\text{ for a.e. }\lambda \in \mathbb{T} \bigr\}\\
&\simeq
\psi_1(\mscr{C}_{\SU_q(2)})\ov{\otimes} \LL^{\infty}(\TT).
\end{split}\]
\end{proposition}

\begin{proof}
Inclusion $\supseteq$ clearly follows from equation \eqref{eq4.1.2}, assume that $T$ belongs to $\mc{Q}_L (\mscr{C}_{\SU_q(2)}'\cap \LL^{\infty}(\SU_q(2)))\mc{Q}_L^*$. Using \eqref{eq4.1.1} we can write $T=\int_{\TT}^{\oplus} T_\lambda\otimes\I_{\ov{\msf{H}_\lambda}}\md\mu(\lambda)$ for some $T_\lambda\in\B(\msf{H}_\lambda)$. Our assumption forces $T_\lambda\in \psi_1(\mscr{C}_{\SU_q(2)})'$ for almost all $\lambda\in\TT$. From the definition of $\psi_1$ we see that $\psi_1(\alpha)$ is a weighted shift and Lemma \ref{lemma4.1.1} applies -- $\psi_1(\mscr{C}_{\SU_q(2)})$ is a MASA in $\B(\ell^2(\ZZ_+))$, hence $T_\lambda\in\psi_1(\mscr{C}_{\SU_q(2)})$ and the claim follows.
\end{proof}

Despite the above negative result, we can nonetheless obtain some examples of MASAs. First of all, since $\psi_1(\mscr{C}_{\SU_q(2)})$ is a MASA in $\B(\ell^2(\ZZ_+))$ we immediately get that the subalgebra generated by $\mscr{C}_{\SU_q(2)}$ and the center of $\LL^{\infty}(\SU_q(2))$ is a MASA in $\LL^{\infty}(\SU_q(2))$. More nontrivially, we will prove that $\mscr{C}_{\SU_q(2)}$ is MASA in the smaller von Neumann algebra of fixed points for the scaling group. Recall that we denote this algebra by $\LL^{\infty}(\SU_q(2))^{\tau}$.

\begin{proposition}\label{prop4.1.2}
The algebra of class functions $\mscr{C}_{\SU_q(2)}$ is a MASA in $\LL^{\infty}(\SU_q(2))^{\tau}$, i.e.
\[
\mscr{C}_{\SU_q(2)}'\cap\LL^{\infty}(\SU_q(2))^{\tau}=\mscr{C}_{\SU_q(2)}.
\]
\end{proposition}

\begin{proof}
Observe first that since $\mscr{C}_{\SU_q(2)}$ is generated by characters, we have $\mscr{C}_{\SU_q(2)}\subseteq \LL^{\infty}(\SU_q(2))^{\tau}$. 
Take $T=\int_{\TT}^{\oplus}T_\lambda\otimes\I_{\ov{\msf{H}_{\lambda}}}\md\mu(\lambda)$ in $\mc{Q}_L(\mscr{C}_{\SU_q(2)}'\cap \LL^{\infty}(\SU_q(2))^{\tau})\mc{Q}_L^*$. Proposition \ref{prop4.1.1} implies that $T_\lambda\in\psi_1(\mscr{C}_{\SU_q(2)})$ for almost all $\lambda\in\TT$. Denote by $P$ the operator implementing the scaling group for $\SU_q(2)$ and its dual \cite[Definition 6.9]{KustermansVaes1}. We know how to express this operator on the level of direct integrals (see \cite[Proposition 7.3]{KrajczokModular}): for all $t\in\RR$ we have
\begin{equation}\label{eq4.1.4}
\mc{Q}_L P^{it} \mc{Q}_L^*\colon \int_{\TT}^{\oplus} \HS(\msf{H}_\lambda)
\md\mu(\lambda)\ni \int_{\TT}^{\oplus}\xi_\lambda \md\mu(\lambda)\mapsto
\int_{\TT}^{\oplus}\xi_{\lambda|q|^{2it}} \md\mu(\lambda)
\in \int_{\TT}^{\oplus} \HS(\msf{H}_\lambda)
\md\mu(\lambda),
\end{equation}
and since $T$ is invariant under the (transformed) scaling group we have
\begin{equation}\label{eq4.1.3}
\int_{\TT}^{\oplus} T_\lambda \otimes \I_{\ov{\msf{H}_\lambda}}\md\mu(\lambda)=T=(\mc{Q}_L P^{it} \mc{Q}_L^*) T 
(\mc{Q}_L P^{-it} \mc{Q}_L^*)=
\int_{\TT}^{\oplus} T_{\lambda|q|^{2it}}\otimes \I_{\ov{\msf{H}_\lambda}}\md\mu(\lambda).
\end{equation}
It follows that $\TT\ni\lambda\mapsto T_\lambda\in \B(\ell^2(\ZZ_+))$ is constant almost everywhere. Indeed, for $\kappa\in\TT$ denote by $f_\kappa\in \B(\ell^2(\ZZ_+))\ov{\otimes} \LL^{\infty}(\TT)$ the function $\lambda\mapsto T_{\lambda \kappa}$. Equation \eqref{eq4.1.3} implies $f_1=f_\kappa\,(\kappa\in\TT)$. For all $\theta\in\LL^1(\TT),\omega\in \B(\ell^2(\ZZ_+))_*$ we get\footnote{Integrals of $\B(\ell^2(\ZZ_+))$ or $\LL^{\infty}(\TT)\ov{\otimes}\B(\ell^2(\ZZ_+))$--valued functions are understood in the sense of Pettis, where the von Neumann algebras are equipped with the $w^*$--topology.}
\[\begin{split}
&\quad\;
(\omega\otimes\theta)f_1=(\omega\otimes\theta)\bigl(\int_{\TT} f_\kappa\md\mu(\kappa)\bigr)=
\int_{\TT} (\omega\otimes\theta) (f_\kappa)\md\mu(\kappa)\\
&=
\int_{\TT} \int_{\TT}\theta(\lambda) \omega(f_{\kappa}(\lambda))\md\mu(\lambda)\md\mu(\kappa)=
\int_{\TT} \theta(\lambda) \int_{\TT}\omega(T_{\lambda\kappa})\md\mu(\kappa)\md\mu(\lambda)\\
&=
\bigl(\int_{\TT}\theta(\lambda)\md\mu(\lambda)\bigr) \int_{\TT} \omega(T_\kappa)\md\mu(\kappa)=
(\omega\otimes\theta)\bigl(\int_{\TT} T_\kappa \md\mu(\kappa)\otimes 
\I_{\LL^{\infty}(\TT)}\bigr),
\end{split}\]
hence $f_1=\int_{\TT} T_\kappa\md\mu(\kappa)\otimes \I_{\LL^{\infty}(\TT)}$. Consequently, $T$ belongs to $\mc{Q}_L \mscr{C}_{\SU_q(2)}\mc{Q}_L^*$ (see equation \eqref{eq4.1.2}).
\end{proof}

\subsection{A certain bicrossed product construction}\label{subsec:bicross}

In this subsection we present a construction of a class of compact quantum groups $\HH$ given by a bicrossed product of a compact quantum group $\GG$ and the additive group of rational numbers $\QQ$ (in this paper we equip $\QQ$ with the discrete topology), where $\QQ$ acts on $\Linf$ using the scaling group of $\GG$. Our construction is a slight variation of a construction presented in \cite[Section 4.1]{DasDawsSalmi} -- the main difference is that we replace $\RR$ with a discrete group $\QQ$ in order to get a compact quantum group as the bicrossed product. The principal reason why we are interested in this family of quantum groups is the fact that they admit nontrivial inner scaling automorphisms -- a property that appeared in Proposition \ref{Prop:scaling} (see Lemma \ref{lemma4.2.1}. Observe also that equation \eqref{eq4.1.4} implies that nontrivial scaling automorphisms of $\SU_q(2)$ are never inner. Another reason is that these bicrossed products provide examples of compact quantum groups $\HH$ with $\LL^{\infty}(\HH)$ being the type $\operatorname{II}_{\infty}$ injective factor.\\

Later on we will specify to $\GG=\SU_q(2)$, for now let $\GG$ be an arbitrary compact quantum group. Fix a nonzero number $\nu\in \RR\setminus\{0\}$ and denote by $\alpha$ the action of $\QQ$ on $\LL^{\infty}(\GG)$ given by 
\[
\alpha_\gamma(x)=\tau_{\nu \gamma}^{\GG}(x)\quad(\gamma\in \QQ, \, x\in \Linf).
\]
Let $\HH=\QQ\bowtie\GG$ be the resulting quantum group (see \cite[Section 6]{bicrossed}. One can also define $\HH$ as a bicrossed product of $\QQ$ and $\GG$, hence the notation \cite[Definition 2.1]{VaesVainerman}). For the details of this construction and its properties we refer the reader to \cite{VaesVainerman, Wang, bicrossed}, here we will recall only some of its aspects. By abuse of notation, let us also denote by $\alpha$ the map $\alpha\colon \Linf\rightarrow  \ell^{\infty}(\QQ)\ov{\otimes} \Linf$ given by $\alpha(x)(\gamma)=\tau^{\GG}_{\nu\gamma}(x)$. We have
\[
\LL^{\infty}(\HH)=\QQ\ltimes_{\alpha}\Linf=\{\alpha(x), u_\gamma\,|\,
x\in\Linf, \gamma\in\QQ\}''
\]
(where $\QQ\ni\gamma\mapsto \lambda_\gamma\in \B(\ell^2(\QQ))$ is the left regular representation, $u_\gamma=\lambda_\gamma\otimes\I$) and
\[
\ell^{\infty}(\widehat{\HH})=\ell^{\infty}(\QQ)\ov{\otimes}\ell^{\infty}(\widehat{\GG}).
\]
These von Neumann algebras are represented on the Hilbert space
\[
\LL^2(\HH)=\ell^{2}(\QQ)\otimes \LdG.
\]
The bicrossed product $\HH$ is compact as it is built from a discrete and a compact quantum group. In fact, the GNS map for $h_{\HH}$ is given by
\begin{equation}\label{eq4.2.1}
\Lambda_{h_{\HH}}(u_\gamma \alpha(x))=\Lambda_{h_{\widehat{\QQ}}}(\lambda_\gamma)\otimes 
\Lambda_{h_{\GG}}(x)\quad(x\in\Linf, \gamma\in\QQ).
\end{equation}
We can also identify the (left) Haar integral on $\whH$ -- it is equal to $\vp_{\QQ}\otimes \vp_{\whG}$ (where $\vp_{\QQ},\vp_{\whG}$ are the left Haar integrals on $\QQ$ and $\whG$), hence
\[
\nabla_{\vp_{\whH}}=\I\otimes \nabla_{\vp_{\whG}}
\]
(it is a combination of Proposition 2.9, Theorem 2.13 and Proposition 2.16 in \cite{VaesVainerman}). Let $P_{\GG}$ be the positive operator implementing scaling group, see \cite[Definition 5.1, Remark 5.2]{Daele}. Since equality $\nabla_{\vp_{\whG}}^{it}=P_{\GG}^{it}$ holds for any unimodular locally compact quantum group (\cite[Proposition 5.6]{Daele}), we arrive at
\begin{equation}\label{eq4.2.2}
P_{\HH}^{it}=\nabla_{\vp_{\whH}}^{it}=
\I\otimes \nabla_{\vp_{\whG}}^{it}=
\I\otimes P_{\GG}^{it}\quad(t\in\RR).
\end{equation}

The irreducible representations of $\HH$ are indexed by $\QQ\times \IrrG$. The corresponding characters and $\uprho$ -- operators are given by
\[
\chi_{(\gamma,\delta)}= u_\gamma\alpha(\chi_\delta),\quad
\uprho_{(\gamma,\delta)}=\uprho_{\delta}\quad( (\gamma,\delta)\in \QQ\times\IrrG)
\]
(see \cite[Theorem 6.1]{bicrossed}). 
It is a well known property of crossed products that automorphisms with which $\QQ$ acts on $\Linf$ become inner after the inclusion of $\Linf$ into $\QQ\ltimes_{\alpha}\Linf$:
\begin{equation}\label{eq4.2.3}
\alpha(\tau^{\GG}_{\nu \gamma}(x))=u_\gamma
\alpha(x)
u_\gamma^*\quad(x\in\Linf,\gamma\in\QQ).
\end{equation}

Let us now record a simple result concerning the scaling group of $\HH$.

\begin{lemma}\label{lemma4.2.1}$ $
\begin{itemize}
\item We have $\tau^{\HH}_{t}(\alpha(x))=\alpha(\tau^{\GG}_{t}(x))$ and $\tau^{\HH}_t(u_\gamma)=u_\gamma$ for all $t\in\RR,x\in\Linf,\gamma\in\QQ$.
\item For every $t\in\RR$, the scaling automorphism $\tau^{\HH}_{t}$ is trivial if and only if $\tau^{\GG}_{t}$  is trivial. If $\gamma\in\QQ$, then $\tau^{\HH}_{\nu\gamma}$ is inner. 
\end{itemize}
\end{lemma}

\begin{proof}
The first part is a direct consequence of equations \eqref{eq4.2.1}, \eqref{eq4.2.2}:
\[\begin{split}
&\quad\;
\Lambda_{h_{\HH}}(\tau_t^{\HH}(\alpha(x)))=
(\I\otimes P_{\GG}^{it})(\Lambda_{h_{\widehat{\QQ}}}(\I)\otimes \Lambda_{h_{\GG}}(x))=
\Lambda_{h_{\widehat{\QQ}}}(\I)\otimes \Lambda_{h_{\GG}}(\tau^{\GG}_{t}(x))=
\Lambda_{h_{\HH}}(\alpha(\tau_t^{\GG}(x)))
\end{split}\]
and
\[
\Lambda_{h_{\HH}}(\tau^{\HH}_t(u_\gamma))=
(\I\otimes P_{\GG}^{it})(\Lambda_{h_{\widehat{\QQ}}}(\lambda_\gamma)\otimes \Lambda_{h_{\GG}}(\I))=
\Lambda_{h_{\widehat{\QQ}}}(\lambda_\gamma)\otimes \Lambda_{h_{\GG}}(\I)=
\Lambda_{h_{\HH}}(u_\gamma).
\]
Since $h_{\HH}$ is faithful on $\LL^{\infty}(\HH)$ we get the first claim. As $\alpha$ is a monomorphism, $\tau^{\HH}_t$ is trivial if and only so is $\tau^{\GG}_{t}$. The last claim follows from equation \eqref{eq4.2.3}.
\end{proof}

Let us end these general considerations with an observation that
\[
u_\gamma\in\mscr{C}_{\HH}\quad(\gamma\in\QQ)\quad \textnormal{and}\quad \alpha(\mscr{C}_{\GG})\subseteq \mscr{C}_{\HH}.
\]
Indeed, it is a consequence of \cite[Theorem 3.7]{Wang}.\\

Fix $q\in (-1,1)\setminus\{0\}$. From now on we consider the special case $\GG=\SU_q(2)$ -- accordingly $\HH$ is given by $\HH=\QQ\bowtie \SU_q(2)$. Note that this quantum group depends on two parameters: $q$ and $\nu$. Using Proposition \ref{prop4.1.2} ($\mscr{C}_{\SU_q(2)}$ is MASA in $\LL^{\infty}(\SU_q(2))^{\tau}$) we are able to deduce the following interesting property of $\HH$:

\begin{proposition}\label{prop4.2.1}
Let $\HH=\QQ\bowtie\SU_q(2)$. The von Neumann algebra $\mscr{C}_{\HH}$ is a MASA in $\LL^{\infty}(\HH)$.
\end{proposition}

\begin{proof}
First, it is clear that $\mscr{C}_{\HH}$ is commutative. Indeed, since $\mscr{C}_{\SU_q(2)}$ is commutative, commutativity of $\mscr{C}_{\HH}$ follows from \cite[Theorem 3.7]{Wang}. Take now $T\in \mscr{C}_{\HH}'\cap \LL^{\infty}(\HH)$ -- we want to show $T\in\mscr{C}_{\HH}$. Let $\EE\colon \LL^{\infty}(\HH)=\QQ\ltimes_{\alpha} \LL^{\infty}(\SU_q(2))\rightarrow \LL^{\infty}(\SU_q(2))$ be the canonical faithful normal conditional expectation satisfying $\EE(u_\gamma \alpha(x))=\delta_{\gamma,0} x$. Define operators
\[
T_\gamma=\EE(u_\gamma^* T)\in\LL^{\infty}(\SU_q(2))\quad(\gamma\in\QQ).
\]

Clearly we have
\begin{equation}\label{eq4.2.5}
\ismaa{\xi}{T_\gamma \eta}=
\ismaa{\xi}{\EE(u_\gamma^* T) \eta}=
\ismaa{\delta_0 \otimes\xi}{ (u_\gamma^* T) (\delta_0\otimes\eta)}
\end{equation}
for all $\xi,\eta\in\LL^2(\SU_q(2))$. 
Fix $\gamma\in\QQ$. Using the fact that $T\in \mscr{C}_{\HH}'\cap \LL^{\infty}(\HH)$ we will now show $T_\gamma\in \mscr{C}_{\SU_q(2)}'$. Since for any $y\in\mscr{C}_{\SU_q(2)}$ operator $\alpha(y)$ belongs to $\mscr{C}_{\HH}$, we get
\[\begin{split}
&\quad\;
\ismaa{\xi}{T_\gamma y\eta}=
\ismaa{\delta_\gamma\otimes\xi}{T (\delta_0 \otimes y\eta)}=
\ismaa{\delta_\gamma\otimes\xi}{T \alpha( y)(\delta_0 \otimes \eta)}=
\ismaa{\alpha(y^*)(\delta_\gamma\otimes \xi)}{T (\delta_0\otimes\eta)}\\
&=
\ismaa{\delta_0\otimes y^*\xi}{(u_\gamma^* T)(\delta_0\otimes\eta)}=\ismaa{y^* \xi}{T_\gamma\eta}=\ismaa{\xi}{y T_\gamma\eta}
\end{split}\]
for all vectors $\xi,\eta\in\LL^2(\SU_q(2))$ and consequently $T_\gamma\in \mscr{C}_{\SU_q(2)}'$.\\
Take $\gamma'\in \QQ$. Observe that Lemma \ref{lemma4.2.1} together with equation \eqref{eq4.2.3} implies that $\tau^{\HH}_{\nu \gamma'}$ is implemeneted by $u_{\gamma'}\in\mscr{C}_{\HH}$. Using equation \ref{eq4.2.2} we calculate
\[\begin{split}
&\quad\;
\ismaa{\xi}{\tau^{\SU_q(2)}_{\nu\gamma'}(T_\gamma) \eta}=
\ismaa{\delta_0 \otimes P^{-i\nu\gamma'}_{\SU_q(2)}\xi}{
(u_\gamma^* T) (\delta_0\otimes P^{-i\nu\gamma'}_{\SU_q(2)}\eta)}=
\ismaa{\delta_\gamma\otimes\xi}{
P_{\HH}^{i\nu\gamma'} T P_{\HH}^{-i\nu\gamma'} (\delta_0\otimes \eta)}\\
&=
\ismaa{\delta_\gamma\otimes\xi}{\tau^{\HH}_{\nu\gamma'}(T)(\delta_0\otimes\eta)}=
\ismaa{\delta_\gamma\otimes\xi}{u_{ \gamma'} T u_{\gamma'}^* (\delta_0\otimes\eta)}=
\ismaa{\delta_\gamma\otimes\xi}{ T (\delta_0\otimes\eta)}=
\ismaa{\xi}{T_\gamma \eta}
\end{split}\]
and as before we arrive at $\tau_{\nu\gamma'}^{\SU_q(2)}(T_\gamma)=T_\gamma$. Density of $\nu\QQ$ in $\RR$ implies $T_\gamma\in\LL^{\infty}(\SU_q(2))^{\tau}$. Alternatively, one can also obtain this result by showing that $\EE$ commutes with the scaling group and $T\in \LL^{\infty}(\HH)^{\tau}$.\\
These two properties of $T_\gamma$ imply that $T_\gamma\in \mscr{C}_{\SU_q(2)}$ (Proposition \ref{prop4.1.2}) and consequently $\alpha(T_\gamma)\in \mscr{C}_{\HH}$. \\

Formally we have $T=\sum_{\gamma\in\QQ} u_\gamma \alpha(T_\gamma)$. However, this series does not need to converge in the $w^*$-topology (see \cite{Mercer}), which is why we will argue on the $\LL^2$-level. Let us first prove that
\begin{equation}\label{eq4.2.4}
\Lambda_{h_{\HH}}(T)=\sum_{\gamma\in\QQ}\Lambda_{h_{\HH}}(u_\gamma \alpha(T_\gamma))=
\sum_{\gamma\in\QQ} \Lambda_{h_{\widehat{\QQ}}}(\lambda_\gamma)\otimes
\Lambda_{h_{\SU_q(2)}}(T_\gamma).
\end{equation}

Since $\{\delta_\gamma\}_{\gamma\in\QQ}$ forms an orthonormal basis in $\ell^2(\QQ)$, we can write $\Lambda_{h_{\HH}}(T)=\sum_{\gamma\in\QQ} \delta_\gamma\otimes \tilde{T}_{\gamma}$ for some $\tilde{T}_{\gamma}\in \LL^2(\SU_q(2))$. Then
\[
\ismaa{\xi}{\tilde{T}_\gamma}=
\ismaa{\delta_\gamma\otimes\xi}{\sum_{\gamma'\in\QQ} \delta_{\gamma'}\otimes \tilde{T}_{\gamma'}}=
\ismaa{\delta_\gamma\otimes\xi}{\Lambda_{h_{\HH}}(T)}=
\ismaa{\delta_\gamma\otimes\xi}{T(\delta_0\otimes \Lambda_{h_{\SU_q(2)}}(\I)}=\ismaa{\xi}{\Lambda_{h_{\SU_q(2)}}(T_\gamma)}
\]
for all $\gamma\in\QQ,\xi\in\LL^2(\SU_q(2))$ which proves \eqref{eq4.2.4}. Recall that $h_{\HH}$ is tracial on $\mscr{C}_{\HH}$, hence the claim follows from equation \eqref{eq4.2.4} and the following lemma.
\end{proof}

\begin{lemma}\label{lemma4.2.2}
Let $(\M,\omega)$ be a von Neumann algebra with a fixed faithful normal state. Assume that $\N\subseteq \M$ is a von Neumann subalgebra such that $\omega|_{\N}$ is tracial. If $x\in\M$ and $\Lambda_\omega(x)\in\ov{\Lambda_{\omega}(N)}$ then $x\in \N$.
\end{lemma}

This lemma is well-known to experts but we were not able to locate a precise reference, so we decided to add a proof for completeness.

\begin{proof}
We will show that $x$ commutes with every $y\in\N'$. Take $a,b\in \M$ that are analytic with respect to $(\sigma_t)_{t\in\RR}$ and fix a net $(\Lambda_\omega(x_i))_{i\in I}\,(x_i\in \N)$ which converges to $\Lambda_\omega(x)$. Observe that since $\omega|_{\N}$ is tracial, $J\nabla^{\frac{1}{2}}$ is an isometry on $\Lambda_\omega(\N)$. As $J\nabla^{\frac{1}{2}}$ is closed, it follows that $\lim_{i\in I}\Lambda_\omega(x_i^*)=\Lambda_\omega(x^*)$. Consequently
\[\begin{split}
&\quad\;
\ismaa{\Lambda_\omega(a)}{ yx \Lambda_\omega(b)}=
\ismaa{\Lambda_\omega(a)}{y J \sigma_{i/2}(b)^* J \Lambda_{\omega}(x)}=
\lim_{i\in I}
\ismaa{\Lambda_\omega(a)}{y J \sigma_{i/2}(b)^* J \Lambda_{\omega}(x_i)}\\
&=
\lim_{i\in I}
\ismaa{\Lambda_\omega(a)}{y x_i \Lambda_\omega(b)}=
\lim_{i \in I}\ismaa{J\sigma_{i/2}(a)^* J
\Lambda_\omega(x_i^*)}{y\Lambda_\omega(b)}=
\ismaa{J \sigma_{i/2}(a)^* J
\Lambda_\omega(x^*)}{y\Lambda_\omega(b)}\\
&=
\ismaa{x^*\Lambda_\omega(a)}{y\Lambda_\omega(b)}=
\ismaa{\Lambda_\omega(a)}{xy \Lambda_\omega(b)}.
\end{split}\]
A standard density argument implies $x\in \N''=\N$.
\end{proof}

\begin{remark}
In the proof of Proposition \ref{prop4.2.1}, we argued on the $\LL^2$-level that $\alpha(T_\gamma)\in\mscr{C}_{\HH}\,(\gamma\in\QQ)$ implies that $T\in\mscr{C}_{\HH}$. Alternatively, we could use a Fej\'er-type theorem for crossed products and arrive at the same conclusion (see e.g.~\cite[Theorem 4.10]{CrannNeufang} for a general result).
\end{remark}

In the penultimate result we prove about $\HH=\QQ\bowtie\SU_q(2)$, we study its von Neumann algebra of bounded functions. In particular, we show that for some values of $\nu,q$, it is a factor of type $\operatorname{II}_\infty$ -- we are not aware of another example of a compact quantum group in the literature with this property.

\begin{proposition}\label{prop4.2.2}$ $
\begin{itemize}
\item $\mc{Z}(\LL^{\infty}(\HH))$ is equal to $\{u_{\gamma}\,|\, \gamma\in\QQ\,\cap\, \tfrac{\pi}{\nu \log(|q|)}\,\ZZ\}''$. In particular, it is trivial if $\nu\log(|q|)\notin \pi\QQ$ and isomorphic to $\LL^{\infty}(\TT)$ otherwise.
\item Let $t\in\RR$. The scaling automorphism $\tau^{\HH}_{ t}$ is trivial if and only if $ t\in\tfrac{\pi}{ \log(|q|)}\ZZ$. It is inner if and only if $t\in\nu\QQ+ \tfrac{\pi}{\log(|q|)}\ZZ$.

\item $\HH$ is coamenable and consequently $\LL^{\infty}(\HH)$ is injective. 
\item If $\nu\log(|q|)\notin\pi \QQ$ then $\LL^{\infty}(\HH)$ is a factor of type $\operatorname{II}_{\infty}$.
\end{itemize}
\end{proposition}

\begin{proof}
Observe first that for all $t\in\RR$, the scaling automorphism $\tau^{\HH}_{t}$ is trivial if and only $\tau^{\SU_q(2)}_{t}$ is trivial (Lemma \ref{lemma4.2.1}) which happens if and only if $t\in\tfrac{\pi}{\log(|q|)}\ZZ$ (equation \eqref{eq4.1.3}).\\
Take $x\in \mc{Z}(\LL^{\infty}(\HH))$. Since $\mscr{C}_{\HH}$ is MASA in $\LL^{\infty}(\HH)$, we know that
\[
x\in \ov{\lin\{\chi_\beta\,|\, \beta\in\Irr(\HH):\uprho_\beta=\I_{\beta}\}}^{\,\sot}=
\ov{\lin \{u_\gamma\,|\, \gamma\in\QQ\}}^{\,\sot}
\]
(Proposition \ref{Prop:masafactor}). Write
\[
x=\sot-\underset{i\in I}{\lim}\sum_{\gamma\in\QQ} C^i_\gamma u_{\gamma},\quad
\Lambda_{h_{\HH}}(x)=\sum_{\gamma\in \QQ} C_\gamma \Lambda_{h_{\HH}}(u_\gamma)
\]
for some $C_\gamma,C^i_\gamma\in\CC$, where $\sum_{\gamma\in\QQ} C^i_\gamma u_\gamma$ belongs to $\lin\{u_\gamma\,|\, \gamma\in\QQ\}$ for each $i\in I$. Take now $y\in \LL^{\infty}(\SU_q(2))$. Since $x\in\mc{Z}(\LL^{\infty}(\HH))$, we have
\[\begin{split}
&\quad\;
\sum_{\gamma\in\QQ} C_\gamma \delta_\gamma\otimes \Lambda_{h_{\SU_q(2)}}
(\tau^{\SU_q(2)}_{\nu \gamma}(y))=
\alpha(y) \bigl(\sum_{\gamma\in\QQ} C_\gamma \Lambda_{h_{\HH}}(u_\gamma)\bigr)=
\Lambda_{h_{\HH}}(\alpha(y) x)=\Lambda_{h_{\HH}}(x\alpha(y))\\
&=
x(\delta_0\otimes \Lambda_{h_{\SU_q(2)}}(y))=
\lim_{i\in I} \sum_{\gamma\in\QQ}C^i_\gamma( \delta_\gamma\otimes\Lambda_{h_{\SU_q(2)}}(y)),
\end{split}\]
which implies
\[
C_\gamma \tau^{\SU_q(2)}_{\nu\gamma}(y)=\lim_{i\in I} C^i_\gamma\, y\quad(\gamma\in \QQ).
\]
As this equation holds for every $y\in \LL^{\infty}(\SU_q(2))$, we must have $C_\gamma=0$ whenever $\tau^{\SU_q(2)}_{\nu\gamma}$ is nontrivial, i.e. for $\nu\gamma\notin \tfrac{\pi}{\log(|q|)}\ZZ$. Lemma \ref{lemma4.2.2} gives us
\[
\mc{Z}(\LL^{\infty}(\HH))\subseteq \{u_\gamma \,|\, \gamma\in\QQ\cap \tfrac{\pi}{\nu\log(|q|)}\ZZ\}''.
\]
Inclusion $\supseteq$ is clear, hence we have identified the center of $\LL^{\infty}(\HH)$. If $\nu\log(|q|)\notin\pi\QQ$ then clearly $\QQ\cap\tfrac{\pi}{\nu\log(|q|)}\ZZ=\{0\}$ and $\LL^{\infty}(\HH)$ is a factor. Otherwise $\QQ\cap \tfrac{\pi}{\nu\log(|q|)}\ZZ$ is a subgroup of $\QQ$ isomorphic to $\ZZ$ and $\{u_\gamma\,|\, \gamma\in\QQ\cap \tfrac{\pi}{\nu\log(|q|)}\ZZ\}''$ is therefore isomorphic to $\LL(\ZZ)\simeq\LL^{\infty}(\TT)$ \cite[Theorem A]{Herz}. This proves the first point\footnote{We could also argue that $\LL^{\infty}(\HH)$ is a factor if $\nu\log(|q|)\notin\pi\QQ$ using \cite[Theorem 7.11.11]{Pedersen}.}.\\

Take now $t\in\RR$. If $t=\nu\gamma+\tfrac{\pi}{\log(|q|)}s\in\nu\QQ+\tfrac{\pi}{\log(|q|)}\ZZ$ then $\tau^{\HH}_{ t}=\tau^{\HH}_{\nu \gamma}\tau^{\HH}_{\pi s/\log(|q|)}=\tau^{\HH}_{\nu\gamma}$ is inner by Lemma \ref{lemma4.2.1}. Assume that $t\notin\nu\QQ+\tfrac{\pi}{\log|q|}\ZZ$ and $\tau^{\HH}_{t}=\operatorname{Ad}_v$ for some unitary $v\in\LL^{\infty}(\HH)$. Proposition \ref{Prop:scaling} implies that $v\in \{u_\gamma\,|\, \gamma\in\QQ\}''$, hence we can write
\begin{equation}\label{eq4.2.6}
\Lambda_{h_{\HH}}(v)=\sum_{\gamma\in\QQ} D_\gamma \Lambda_{h_{\HH}}(u_\gamma)=
\sum_{\gamma\in\QQ} D_\gamma (\delta_\gamma\otimes\Lambda_{h_{\SU_q(2)}}(\I))
\end{equation}
for some $D_\gamma\in\CC$. Since $v$ is unitary, we have $\sum_{\gamma\in\QQ} |D_\gamma|^2=1$. Let $f\in\LL^{\infty}(\TT)$ be the characteristic function of the arc $\{e^{i\theta}\,|\, \theta\in [0,\pi]\}\subseteq\TT$ and $F=\mc{Q}_L^*(\int_{\TT}^{\oplus}f(\lambda)\I_{\HS(\msf{H}_\lambda)}\md\mu(\lambda))\mc{Q}_L\in\LL^{\infty}(\SU_q(2))$. Equation \eqref{eq4.2.6} together with Lemma \ref{lemma4.2.1} gives us
\[\begin{split}
&\quad\;\sum_{\gamma\in\QQ} D_\gamma ( \delta_\gamma \otimes \Lambda_{h_{\SU_q(2)}}(\tau^{\SU_q(2)}_{\nu \gamma}(F)))
=
\alpha(F)\bigl( \sum_{\gamma\in\QQ} D_\gamma(\delta_\gamma\otimes \Lambda_{h_{\SU_q(2)}}(\I))\bigr)=
\Lambda_{h_{\HH}}(\alpha(F) v)\\
&=
v\Lambda_{h_{\HH}}(\tau^{\HH}_{- t}(\alpha(F)))=
v\bigl(  \delta_0\otimes \Lambda_{h_{\SU_q(2)}}(\tau^{\SU_q(2)}_{- t}(F))\bigr).
\end{split}\]
Since $v\in \{\lambda_{\gamma}\otimes\I\,|\, \gamma\in\QQ\}''$, the last vector belongs to $\ov{\lin} \{\delta_{\gamma}\otimes \Lambda_{h_{\SU_q(2)}}(\tau^{\SU_q(2)}_{- t}(F))\,|\,\gamma\in\QQ\}$. It follows that there exists $\gamma\in \QQ$ such that
\[
\tau^{\SU_q(2)}_{\nu \gamma}(F)=c \tau^{\SU_q(2)}_{- t}(F)
\]
for some $c\in \CC$. Each scaling automorphism acts by a rotation (equation \eqref{eq4.1.3}), hence $c=1$ and $\tau^{\SU_q(2)}_{t+\nu \gamma}(F)=F$. However, $\tau^{\SU_q(2)}_{t+\nu \gamma}$ is a nontrivial rotation. Indeed, otherwise $t+\nu\gamma \in \tfrac{\pi}{\log(|q|)}\ZZ$ and we assume that it is not the case. It follows that $f$ is equal to its proper rotation, a contradiction. This ends the proof of the second bullet.\\

The compact quantum group $\HH=\QQ\bowtie \SU_q(2)$ is coamenable because $\QQ$ is amenable and $\SU_q(2)$ is coamenable \cite[Theorem 15]{amenability}. It follows that $\LL^{\infty}(\HH)$ is injective \cite[Theorem 3.3]{coamenability}. Alternatively, to obtain injectivity of $\LL^{\infty}(\HH)$ one can also use the fact that a crossed product of an injective von Neumann algebra by an action of an amenable group is injective \cite[Theorem 3.16]{TakesakiIII}.\\

Assume $\nu\log(|q|)\notin \pi\QQ$. We already know that $\LL^{\infty}(\HH)=\QQ\ltimes_{\alpha}\LL^{\infty}(\SU_q(2))$ is a factor. Since the n.s.f.~tracial weight on $\LL^{\infty}(\SU_q(2))\simeq  \B(\ell^2(\ZZ_+))\ov{\otimes} \LL^{\infty}(\TT)$ given by $\Tr\otimes h_{\TT}$ is invariant under the action of $\QQ$, it gives rise to a n.s.f.~tracial weight on $\LL^{\infty}(\HH)$ (\cite[Theorem 1.17]{TakesakiII}) and consequently $\LL^{\infty}(\HH)$ is not of type $\operatorname{III}$. It follows from the proof of \cite[Theorem 1.3]{cqg} that if there was a faithful normal tracial state on $\LL^{\infty}(\HH)$, then $\HH$ would be of Kac type. As this is not the case, $\LL^{\infty}(\HH)$ cannot be of type $\operatorname{II}_{1}$; we are left with two cases, $\operatorname{I}_{\infty}$ and $\operatorname{II}_{\infty}$. Clearly $|\nu \QQ + \tfrac{\pi}{\log(|q|)}\ZZ|=\aleph_0<|\RR|$ hence there exists a scaling automorphism $\tau^{\HH}_{t}$ which is not inner. It is well known that all automorphisms of $\B(\ell^2(\ZZ_+))$ are inner (\cite[II.5.5.14]{Blackadar}), hence $\LL^{\infty}(\HH)$ has to be of type $\operatorname{II}_{\infty}$.\\
Let us also give an alternative proof of the result that $\LL^{\infty}(\HH)$ is not of type $\operatorname{I}_{\infty}$. Let
\[
\EE\colon \LL^{\infty}(\HH)=\QQ\ltimes_{\alpha} \LL^{\infty}(\SU_q(2))\rightarrow \LL^{\infty}(\SU_q(2))
\]
be the canonical faithful normal conditional expectation. Assume by contradiction that $\LL^{\infty}(\HH)$ is of type $\operatorname{I}_{\infty}$. Then it is purely atomic and it follows that $\LL^{\infty}(\SU_q(2))\simeq \B(\ell^2(\ZZ_+))\ov{\otimes} \LL^{\infty}(\TT)$ is purely atomic as well (\cite[Theorem IV.2.2.4]{Blackadar}), which gives us a contradiction.

Yet another way to prove that $\LL^{\infty}(\HH)$ is a factor of type $\operatorname{II}_{\infty}$ is to use properties of crossed products. Indeed, we know that $\LL^{\infty}(\SU_q(2)) \simeq \B(\ell^{2}(\mathbb{Z}_{+})) \overline{\otimes} \LL^{\infty}(\mathbb{T})$ and by formula \ref{eq4.1.4} and the scaling group acts trivially on the first factor and by rotations on the second factor, hence $\LL^{\infty}(\HH) \simeq \B(\ell^{2}(\mathbb{Z}_{+})) \overline{\otimes} \left(  \LL^{\infty}(\mathbb{T}) \rtimes \mathbb{Q}\right)$. If $\nu \log(|q|) \notin \pi \mathbb{Q}$ then all these rotations are irrational, therefore the action is free, ergodic, and it preserves the Lebesgue measure on the unit circle. By \cite[Corollary 7.8]{TakesakiI}  $\LL^{\infty}(\mathbb{T}) \rtimes \mathbb{Q}$ is a  factor of type $\operatorname{II}_1$, hence $\LL^{\infty}(\HH)$ is a factor of type $\operatorname{II}_{\infty}$.
\end{proof}

As a corollary, we can show that our family of bicrossed products contains uncountably many different isomorphism classes of quantum groups. To formulate this result, let us denote by $\HH_{\nu,q}$ the bicrossed product $\QQ\bowtie \SU_q(2)$ constructed using parameter $\nu$.

\begin{corollary}
Let $\nu,\nu'\in \RR\setminus\{0\}, q,q'\in (-1,1)\setminus \{0\}$. If $\HH_{\nu,q}$ and $\HH_{\nu',q'}$ are isomorphic, then $|q|=|q'|$ and $\nu \QQ +\tfrac{\pi}{\log(|q|)}\ZZ=\nu'\QQ + \tfrac{\pi}{\log(|q|)}\ZZ$. In particular, for each $q\in (-1,1)\setminus\{0\}$ the family $\{ \HH_{\nu,q}\,|\, \nu\in \RR\setminus \{0\}\}$ consists of $\mathfrak{c}$ isomorphism classes of compact quantum groups.
\end{corollary}

\begin{proof}
Let $\phi\colon \mathrm{C}(\HH_{\nu,q})\rightarrow \mathrm{C}(\HH_{\nu',q'})$ be a Hopf $\star$-isomorphism implementing the isomorphism between $\HH_{\nu,q}$ and $\HH_{\nu',q'}$ (recall that $\HH_{\nu,q}$ is coamenable). Since $\phi$ intertwines scaling groups (\cite[Proposition 3.15]{Hom}) it follows that for each $t\in\RR$, $\tau^{\HH_{\nu,q}}_t$ is trivial if and only if $\tau^{\HH_{\nu',q'}}_t$  is trivial and consequently Proposition \ref{prop4.2.2} implies $\tfrac{\pi}{\log(|q|)}\ZZ=\tfrac{\pi}{\log(|q'|)}\ZZ\Rightarrow |q|=|q'|$. Next, since inner scaling automorphisms of $\HH_{\nu,q}$ are implemented by elements of $\mathrm{C}(\HH_{\nu,q})$ (similarly for $\HH_{\nu',q'}$) it follows from the same proposition that $\nu\QQ+\tfrac{\pi}{\log(|q|)}\ZZ=\nu'\QQ + \tfrac{\pi}{\log(|q|)} \ZZ$. The last claim is a consequence of $\dim_{\QQ}( \RR/(\QQ \tfrac{\pi}{\log(|q|)}))=\mathfrak{c}$.
\end{proof}

\subsection{Examples with commutative $\mscr{C}_{\GG}$}\label{subsec:abelian}

In this subsection we will prove that the condition $\sum_{\alpha} \sqrt{\tfrac{\dim(\alpha)}{\dim_q(\alpha)}}<+\infty$ from Theorem \ref{Thm:quasi-split} holds for a certain class of non-Kac type compact quantum groups. More precisely, in this subsection we consider any compact quantum group $\GG$ with the following properties:
\begin{enumerate}
\item there exists an irreducible fundamental representation $U$ with $\dim_q(U)>\dim(U)$ and $\ov{U}\simeq U$,
\item irreducible representations of $\GG$ are labeled by $\ZZ_+$, so that $\IrrG=\{U^n\}_{n\in\ZZ_+}$, where $U^1=U$ and $U^0$ is the trivial representation,
\item we have $U^1\tp U^n\simeq \bigoplus_{k=0}^{n+1} C(k,n)U^k$, with $C(n+1,n)\ge 1$ and $\sup_{n\in\ZZ_+} C(n+1,n)<+\infty$.
\end{enumerate}

Let us mention two classes of compact quantum groups that fit into the above description:

\begin{itemize}
\item
Non-Kac type free orthogonal quantum group $\GG=O_F^+$ satisfies the above conditions with $U$ being the standard fundamental representation (see \cite{BanicaO(n)}).
\item Let $(B,\psi)$ be a finite dimensional \cst-algebra with a non-tracial $\delta$-form. The non-Kac type quantum automorphism group $\GG_{Aut}(B,\psi)$ also satisfies the above conditions (see \cite{BanicaFuss}).
\end{itemize}

To keep the notation lighter, let us write $\dim(n)=\dim(U^n)$ and $\dim_q(n)=\dim_q(U^n)$ for all $n\in\ZZ_+$. Using our assumptions on the representation theory of $\GG$ we can show that $(\tfrac{\dim(n)}{\dim_q(n)})_{n\in\ZZ_+}$ decays at an exponential rate.

\begin{lemma}$ $
\begin{itemize}
\item We have $\ov{U^n}\simeq U^n$ and $U^n\tp U^m\simeq U^m\tp U^n$ for all $n,m\in\ZZ_+$. 
\item There exists $d>0,c>1$ such that $\tfrac{ \dim(n)}{\dim_q(n)}\le \tfrac{d}{c^n}$ for all $n\in\ZZ_+$.
\end{itemize}
\end{lemma}

Let us note that when $\GG=O_F^+$ or $\GG=\GG_{Aut}(B,\psi)$ then numbers $\dim(n),\dim_q(n)$ are known, see \cite{BanicaO(n), Symmetries}.

\begin{proof}
Observe that we have $(U^1)^{\stp n}\simeq \bigoplus_{k=0}^{n+1} c'_{k,n} U^k$ for some $c'_{k,n}\in \ZZ_+$. As $\ov{U^1}\simeq U^1$, it follows inductively that $\overline{U^n}\simeq U^n$ for all $n\in\ZZ_+$. Equivalence $U^n \tp U^m\simeq U^m\tp U^n$ can now be justified with the following calculations
\[
U^n\tp U^m\simeq \ov{U^n \tp U^m}\simeq
\ov{U^m}\tp \ov{U^n}\simeq U^m\tp U^n.
\]
To prove the second bullet, let us introduce positive numbers $A_n\ge 1$ via $\dim_q(n)=A_n \dim(n)\,(n\in\ZZ_+)$. Clearly $A_0=1$ and we assume that $A_1>1$. The fusion rule $U^1\tp U^n\simeq \bigoplus_{k=0}^{n+1} C(k,n) U^k$ implies
\[
A_1 A_n \dim(1) \dim(n)=\dim_q(U^1\tp U^n)=
\dim_q\bigl(\bigoplus_{k=0}^{n+1} C(k,n) U^k\bigr)=\sum_{k=0}^{n+1}C(k,n) A_k \dim(k)
\]
and
\[
\dim(1) \dim(n)=\dim(U^1\tp U^n)=
\dim\bigl(\bigoplus_{k=0}^{n+1} C(k,n) U^k\bigr)=\sum_{k=0}^{n+1}C(k,n) \dim(k).
\]
Combining these equations gives us
\[\begin{split}
&\quad\;
A_1 A_n \dim(1) \dim(n)\le 
(\max_{k\in\{0,\dotsc,n\}} A_k) \sum_{k=0}^{n} C(k,n)\dim(k)+
C(n+1,n) A_{n+1} \dim(n+1)\\
&=
(\max_{k\in\{0,\dotsc,n\}} A_k)\bigl(\dim(1)\dim(n)-C(n+1,n)\dim(n+1)\bigr) + C(n+1,n) A_{n+1} \dim(n+1),
\end{split}\]
hence
\[\begin{split}
A_{n+1} &\ge 
\max_{k\in\{0,\dotsc,n\}} A_k + \bigl(A_1 A_n-\max_{k\in\{0,\dotsc,n\}} A_k\bigr) \tfrac{\dim(1)\dim(n)}{C(n+1,n) \dim(n+1)}.
\end{split}\]
The above inequality implies $A_{n+1}=\max_{p\in\{0,\dotsc,n+1\}} A_p$. Consequently, we can further write
\[\begin{split}
A_{n+1}\ge A_n + A_n \tfrac{A_1-1}{\sup_{m\in \ZZ_+} C(m+1,m)} \tfrac{\dim(1) \dim(n)}{\dim(n+1)}.
\end{split}\]
Since $U^{n+1}$ is a subrepresentation of $U^1\tp U^n$, we have $\dim(1)\dim(n)\ge \dim(n+1)$ and
\[
A_{n+1} \ge A_n\bigl(1+ \tfrac{A_1-1}{\sup_{m\in\ZZ_+} C(m+1,m)}\bigr).
\]
Write $c=1+\tfrac{A_1-1}{\sup_{m\in\ZZ_+} C(m+1,m)}>1$. We have shown $A_{n+1}\ge c A_n$. Using $\dim_q(n)=A_n \dim(n)$ we arrive at
\[
\dim_q(n)=A_n \dim(n)\ge c^{n-1} A_1 \dim(n)\quad(n\in\NN).
\]
\end{proof}

In particular, the above lemma implies that $\mscr{C}_{\GG}$ is an abelian von Neumann algebra. Theorems \ref{Thm:quasi-split} and \ref{Thm:nonmasa} give us the following corollary (it follows from the fusion rules that the assumptions are satisfied).

\begin{corollary}\label{Cor:quasisplit}
We have $\sum_{n=0}^{\infty} \sqrt{\tfrac{\dim(n)}{\dim_q(n)}}<+\infty$, hence the inclusion $\mscr{C}_{\GG}\subseteq \Linf$ is quasi-split. Furthermore, $\mscr{C}_{\GG}$ is not MASA.
\end{corollary}

\subsection{Quantum unitary group $U_F^+$}

Let $F$ be an invertible matrix with complex entries and $U_F^+$ the associated compact quantum group. In this subsection we show that the sum condition 
\begin{equation}\label{eq4.4.1}
\sum_{\gamma\in\Irr(U_F^+)}\bigl(\tfrac{\dim(\gamma)}{\dim_q(\gamma)}\bigr)^{\frac{1}{2}}<+\infty
\end{equation}
is satisfied provided $U_F^+$ is ``sufficiently non-Kac'' (see Proposition \ref{prop4.4.1} for a precise result). Consequently, in this case we obtain information about the inclusion $\mscr{C}_{U_F^+}\subseteq\LL^{\infty}(U_F^+)$.\\
The representation theory of $U_F^+$ was described by Banica in \cite[Th{\'e}or{\`e}me 1]{BanicaUn}, let us recall some of its aspects. We can identify $\Irr(U_F^+)$ with the free product of monoids $\ZZ_+\star \ZZ_+$; let $\alpha,\beta$ be the generators and $e$ the neutral element. Then $\alpha$ corresponds to the fundamental representation, $\beta$ to its conjugate and $e$ to the trivial representation. The unique antimultiplicative involution $\ZZ_+\star\ZZ_+\ni x\mapsto \ov{x}\in\ZZ_+\star\ZZ_+$ satisfying $\ov{\alpha}=\beta,\ov{e}=e$ corresponds to the conjugation of representations. Finally, the tensor product of representations is given by
\begin{equation}\label{eq4.4.2}
x\tp y\simeq \bigoplus_{\overset{a,b,c\in \ZZ_+\star \ZZ_+:}{
x=ac,\;y=\ov{c}b}}ab\quad(x,y\in\ZZ_+\star\ZZ_+).
\end{equation}
In order to efficiently calculate the sum \eqref{eq4.4.1}, we need to single out a family of irreducible representations out of which all of $\Irr(U_F^+)$ is built. Observe that each nontrivial word $\gamma\in \Irr(U_F^+)\setminus\{e\}$ has a well defined beginning and an end $s(\gamma),t(\gamma)\in\{\alpha,\beta\}$, e.g. $s(\alpha\beta)=\beta$. Let us define sets
\[\begin{split}
I_{\alpha}&=\{\alpha (\beta\alpha)^n\,|\, n\in\ZZ_+\}\cup \{(\alpha\beta)^n \,|\, n\in\NN\},\\
I_{\beta}&=\{\beta(\alpha\beta)^n \,|\, n\in\ZZ_+\}\cup \{(\beta\alpha)^n\,|\, n\in\NN\}.
\end{split}\]
The following observation was already made e.g.~in \cite{NeshveyevMalacarne}:
\begin{lemma}\label{lemma4.4.1}
Every nontrivial word $\gamma\in \Irr(U_F^+)\setminus\{e\}$ can be uniquely written as
\[
\gamma=x_1\cdots x_p=x_1\tp \cdots\tp x_p
\]
for some $p\in\NN$, $x_1,\dotsc, x_p\in I_\alpha\cup I_\beta$ such that $s(x_k)=t(x_{k+1})$ for $1\le k \le p-1$.
\end{lemma}

The above result follows easily from the observation that if $\delta\alpha^{n}\delta'$ for some $\delta,\delta'\in \Irr(U_F^+)$ and $n\ge 2$ then \eqref{eq4.4.2} implies 
\[
\delta\alpha^n\delta'=\delta\alpha \tp \alpha^{n-1}\delta'
\]
(and similarly for $\delta\beta^n\delta'$). It follows that in order to calculate the sum \eqref{eq4.4.1} we need to find a (quantum) dimension of representations from the sets $I_{\alpha},I_\beta$.

Let us introduce the q-numbers $\left[n\right]_q=\tfrac{q^{-n} - q^{n} }{q^{-1}-q}$, where $n\in\ZZ_+$ and $0<q<1$ \cite{KlimykSchmudgen}. Furthermore, for $n\in\NN$  we will write $w_n^\gamma=\gamma \ov{\gamma} \gamma\cdots$ ($n$ letters), where $\gamma=\alpha$ or $\gamma=\beta$. Thus $t(w^\gamma_n)=\gamma$ and $s(w^{\gamma}_n)$ is equal to $\gamma$ if $n$ is odd and equal to $\ov{\gamma}$ if $n$ is even. We also define $w^\alpha_0=w^\beta_0=e$.

\begin{lemma}\label{lemma4.4.2}
Let $d$ be the classical or the quantum dimension function, $\gamma\in\{\alpha,\beta\}$ and $n\in\ZZ_+$.  If $d(\alpha)=2$, then $d(w^\gamma_n)=n+1$. If $d(\alpha)>2$, then $d(w^\gamma_n)=[n+1]_{q}$, where $0<q<1$ is chosen so that $q^{-1}+q=d(\alpha)$.
\end{lemma}

\begin{proof}
Fix $n\in\ZZ_+$. As
\[
(\alpha\beta)^n\alpha\tp\beta=
(\alpha\beta)^{n+1}\oplus (\alpha\beta)^n,
\]
we have
\begin{equation}\label{eq4.4.3}
d(\alpha) d((\alpha\beta)^n\alpha)=
d((\alpha\beta)^{n+1})+d((\alpha\beta)^n).
\end{equation}
Similarly,
\[
(\alpha\beta)^n\alpha \beta\tp \alpha=(\alpha\beta)^{n+1}\alpha\oplus (\alpha\beta)^n\alpha
\]
and
\begin{equation}\label{eq.4.4.4}
d(\alpha) d((\alpha\beta)^{n+1})=
d((\alpha\beta)^{n+1}\alpha)+ d((\alpha\beta)^n \alpha).
\end{equation}
The formula for $d(w^\alpha_n)$ follows from the following elementary claim which can be proven by induction.\\

If $(a_n)_{n\in\ZZ_+}$ is a sequence in $\RR_{\ge 0}$ such that $a_0=1, a_1\ge 2$ and $a_1 a_n=a_{n-1}+a_{n+1}$ for all $n\in\NN$, then
\[
a_n=\begin{cases}
n+1 & a_1=2 \\
[n+1]_q & a_1>2
\end{cases}
\]
where $0<q<1$ is such that $q^{-1}+q=a_1$.\\

The formula for $d(w^\beta_n)$ can be proven analogously using $d(\alpha)=d(\beta)$.
\end{proof}

Using the above result we can show that for ``sufficiently non-Kac'' quantum unitary groups, the sum condition \eqref{eq4.4.1} is satisfied.

\begin{proposition}\label{prop4.4.1}
Let $0<q_{\oon{q}}\le q_{\oon{c}}\le 1$ be given by $q_{\oon{c}}^{-1}+q_{\oon{c}}=\dim(\alpha)$ and $q_{\oon{q}}^{-1}+q_{\oon{q}}=\dim_q(\alpha)$. If $\dim(\alpha)=2$ and $q_{\oon{q}}<0.0861$ or $\dim(\alpha)\ge 3$ and $\tfrac{q_{\oon{q}}}{q_{\oon{c}}}<\bigl(1+ \sqrt{ \tfrac{3\sqrt{5}+5}{10}}\,\bigr)^{-2}=0.2306\dotsc$, then $\sum_{\gamma\in\Irr(U_F^+)}\sqrt{\tfrac{\dim(\gamma)}{\dim_q(\gamma)}}<+\infty$.
\end{proposition}

\begin{proof}
Lemma \ref{lemma4.4.1} shows
\begin{equation}\begin{split}\label{eq4.4.6}
&\sum_{\gamma\in \Irr(U_F^+)} \bigl(\tfrac{\dim(\gamma)}{\dim_q(\gamma)}\bigr)^{\frac{1}{2}}=1+
\sum_{\delta\in\{\alpha,\beta\}}
\sum_{p=1}^{\infty}
\sum_{\gamma_1\in I_{\delta}}\sum_{\gamma_2\in I_{s(\gamma_1)}}\cdots
\sum_{\gamma_p\in I_{s(\gamma_{p-1})}}
\bigl(\tfrac{\dim(\gamma_1\stp \cdots \stp \gamma_p)}{\dim_q(\gamma_1\stp\cdots\stp \gamma_p)}\bigr)^{\frac{1}{2}}.
\end{split}
\end{equation}
We have
\begin{equation}\label{eq4.4.12}
\sum_{\gamma\in I_\alpha}
\bigl(\tfrac{\dim(\gamma)}{\dim_q(\gamma)}\bigr)^{\frac{1}{2}}=
\sum_{n=1}^{\infty} \bigl(\tfrac{\dim(w^\alpha_n)}{\dim_q(w^\alpha_n)}\bigr)^{\frac{1}{2}},\quad
\sum_{\gamma\in I_\beta}
\bigl(\tfrac{\dim(\gamma)}{\dim_q(\gamma)}\bigr)^{\frac{1}{2}}=
\sum_{n=1}^{\infty} \bigl(\tfrac{\dim(w^\beta_n)}{\dim_q(w^\beta_n)}\bigr)^{\frac{1}{2}},
\end{equation}
consequently Lemma \ref{lemma4.4.2} implies
\[
S:=\sum_{\gamma\in I_\alpha}\bigl(\tfrac{\dim(\gamma)}{\dim_q(\gamma)}\bigr)^{\frac{1}{2}}=
\sum_{\gamma\in I_\beta}\bigl(\tfrac{\dim(\gamma)}{\dim_q(\gamma)}\bigr)^{\frac{1}{2}}.
\]
It follows from equation \eqref{eq4.4.6} that
\[
\sum_{\gamma\in \Irr(U_F^+)}
\bigl(\tfrac{\dim(\gamma)}{\dim_q(\gamma)}\bigr)^{\frac{1}{2}}=
1+\sum_{\delta\in\{\alpha,\beta\}}\sum_{p=1}^{\infty}
S^p=1+2\sum_{p=1}^{\infty} S^p
\]
thus $\sum_{\gamma\in\Irr(U_F^+)}\sqrt{\tfrac{\dim(\gamma)}{\dim_q(\gamma)}}<+\infty$ if, and only if $S<1$.

Let us first consider the case $\dim(\alpha)\ge 3$. Note that in this case $q_{\oon{c}}<1$. Using equation \eqref{eq4.4.12} and Lemma \ref{lemma4.4.2} we calculate
\[\begin{split}
S&=\sum_{n=1}^{\infty} \bigl(\tfrac{ [n+1]_{q_{\oon{c}}}}{[n+1]_{q_{\oon{q}}}}\bigr)^{\frac{1}{2}}=
\sum_{n=2}^{\infty}\bigl( \tfrac{q_{\oon{c}}^{-n} - q_{\oon{c}}^n }{q_{\oon{q}}^{-n}-q_{\oon{q}}^{n}}\,
\tfrac{ q_{\oon{q}}^{-1} - q_{\oon{q}}}{q_{\oon{c}}^{-1} - q_{\oon{c}}}\bigr)^{\frac{1}{2}}=
\bigl( \tfrac{ q_{\oon{q}}^{-1} - q_{\oon{q}}}{q_{\oon{c}}^{-1} - q_{\oon{c}}} \bigr)^{\frac{1}{2}}
\sum_{n=2}^{\infty}
(\tfrac{q_{\oon{c}}}{q_{\oon{q}}})^{-\frac{n}{2}}
\bigl( \tfrac{1 - q_{\oon{c}}^{2n} }{1-q_{\oon{q}}^{2n}}\bigr)^{\frac{1}{2}}\\
&\le
(\tfrac{q_{\oon{q}}}{q_{\oon{c}}})^{-\frac{1}{2}}
\bigl( \tfrac{ 1 - q_{\oon{q}}^2}{1 - q_{\oon{c}}^2} \bigr)^{\frac{1}{2}}
\tfrac{1}{( 1-q_{\oon{q}}^{4})^{\frac{1}{2}}}
\sum_{n=2}^{\infty} 
(\tfrac{q_{\oon{c}}}{q_{\oon{q}}})^{-\frac{n}{2}}=
\tfrac{1}{(1-q_{\oon{c}}^2)^{\frac{1}{2}} (1+q_{\oon{q}}^2)^{\frac{1}{2}}}\; (\tfrac{q_{\oon{c}}}{q_{\oon{q}}})^{\frac{1}{2}} \; \bigl( -1 - (\tfrac{q_{\oon{q}}}{q_{\oon{c}}})^{\frac{1}{2}} + \tfrac{1}{1-(\frac{q_{\oon{q}}}{q_{\oon{c}}})^{\frac{1}{2}}}\bigr)\\
&=
\tfrac{1}{(1-q_{\oon{c}}^2)^{\frac{1}{2}} (1+q_{\oon{q}}^2)^{\frac{1}{2}}}\; (\tfrac{q_{\oon{c}}}{q_{\oon{q}}})^{\frac{1}{2}}\;
\tfrac{ 1- (1- \frac{q_{\oon{q}}}{q_{\oon{c}}})}{1-(\frac{q_{\oon{q}}}{q_{\oon{c}}})^{\frac{1}{2}}}=
\tfrac{1}{(1-q_{\oon{c}}^2)^{\frac{1}{2}} (1+q_{\oon{q}}^2)^{\frac{1}{2}}}\;
\tfrac{(\frac{q_{\oon{q}}}{q_{\oon{c}}})^{\frac{1}{2}}}{1-(\frac{q_{\oon{q}}}{q_{\oon{c}}})^{\frac{1}{2}}}.
\end{split}\]
In the above inequality we have used $\tfrac{1-q_{\oon{c}}^{2n}}{1-q_{\oon{q}}^{2n}}\le \tfrac{1}{1-q_{\oon{q}}^4}$ for all $n\ge 2$. Now, as $\dim(\alpha)\ge 3$, we have $q_{\oon{c}}\le(3-\sqrt{5})/2$ and since $\tfrac{q_{\oon{q}}}{q_{\oon{c}}}<\bigl(1+ \sqrt{ \tfrac{3\sqrt{5}+5}{10}}\,\bigr)^{-2}$,
\[
S< \tfrac{2}{(4- (3-\sqrt{5})^2)^{\frac{1}{2}}}
\;\tfrac{\bigl(1+ \sqrt{ \frac{3\sqrt{5}+5}{10}}\,\bigr)^{-1}}{1-\bigl(1+ \sqrt{ \frac{3\sqrt{5}+5}{10}}\,\bigr)^{-1}}=
\tfrac{2}{(6\sqrt{5}-10)^{\frac{1}{2}}}\;
\tfrac{1}{
\bigl(1+ \sqrt{ \frac{3\sqrt{5}+5}{10}}\bigr)-1
}=
\bigl( 
\tfrac{3\sqrt{5} + 5}{10}
\bigr)^{\frac{1}{2}}
\tfrac{1}{
\sqrt{ \frac{3\sqrt{5}+5}{10}}
}=1
\]
which proves the claim.\\

Let us now consider the case $\dim(\alpha)=2$. We calculate using again Lemma \ref{lemma4.4.2}
\[\begin{split}
S&=\sum_{n=1}^{\infty} \bigl( \tfrac{ n+1}{[n+1]_{q_{\oon{q}}}}\bigr)^{\frac{1}{2}}=
\sum_{n=2}^{\infty}\bigl(\tfrac{n}{q^{-n}_{\oon{q}} - q^n_{\oon{q}}} \,
(q^{-1}_{\oon{q}}-q_{\oon{q}})\bigr)^{\frac{1}{2}}=
(q_{\oon{q}}^{-1} - q_{\oon{q}})^{\frac{1}{2}}\sum_{n=2}^{\infty}
(n q_{\oon{q}}^{n})^{\frac{1}{2}}\,\tfrac{1}{(1-q_{\oon{q}}^{2n})^{\frac{1}{2}}}\\
&\le
q_{\oon{q}}^{-\frac{1}{2}}
\bigl(\tfrac{1-q_{\oon{q}}^2}{1-q_{\oon{q}}^4}\bigr)^{\frac{1}{2}}
\bigl( - q_{\oon{q}}^{\frac{1}{2}} + \sum_{n=1}^{\infty} n q_{\oon{q}}^{\frac{n}{2}}\bigr)=
\tfrac{q_{\oon{q}}^{-\frac{1}{2}}}{(1+q_{\oon{q}}^2)^{\frac{1}{2}}}\bigl(-q_{\oon{q}}^{\frac{1}{2}} + \tfrac{q_{\oon{q}}^{\frac{1}{2}}}{(1-q_{\oon{q}}^{\frac{1}{2}})^2}\bigr)=
\tfrac{ 1-(1-q_{\oon{q}}^{\frac{1}{2}})^2 }{ (1+q_{\oon{q}}^{2})^{\frac{1}{2}} (1-q_{\oon{q}}^{\frac{1}{2}})^2}=
\tfrac{q_{\oon{q}}^{\frac{1}{2}}(2-q_{\oon{q}}^{\frac{1}{2}}) }{ (1+q_{\oon{q}}^{2})^{\frac{1}{2}} (1-q_{\oon{q}}^{\frac{1}{2}})^2}.
\end{split}\]
One can check that the above expression is increasing for $0<q_{\oon{q}}<1$ and smaller than $1$ when $q_{\oon{q}}<0.0861$.
\end{proof}

\begin{remark}\label{Rem:UFnonexample}
Calculations in the proof of Proposition \ref{prop4.4.1} are far from optimal, however it is clear that there are non-Kac type quantum unitary groups $U_F^+$ with $\sum_{\gamma\in\Irr(U_F^+)} \bigl(\tfrac{\dim(\gamma)}{\dim_q(\gamma)}\bigr)^{\frac{1}{2}}=+\infty$. Indeed, for $\dim(\alpha)=2$ we have $S=\sum_{n=2}^{\infty} (\tfrac{n}{[n]_{q_{\oon{q}}}})^{\frac{1}{2}}> (2/[2]_{q_{\oon{q}}})^{\frac{1}{2}}+(3/[3]_{q_{\oon{q}}})^{\frac{1}{2}}$ hence $S>1$ when $q_{\oon{q}}>0.2134$.
\end{remark}

It follows from the rule \eqref{eq4.4.2} and Lemma \ref{lemma4.4.2} that if $U_F^+$ is not of Kac type, the nontrivial irreducible representations $\gamma$ of $U_F^+$ have $\uprho_{\gamma}\neq \I_{\gamma}$. Using Theorem \ref{Thm:quasi-split} we get the following corollary.

\begin{corollary}\label{Cor:UF}
Assume that $\dim(\alpha)=2$ and $q_{\oon{q}}< 0.0861$ or $\dim(\alpha)\ge 3$ and $\tfrac{q_{\oon{q}}}{q_{\oon{c}}}<\bigl(1+\sqrt{\tfrac{3\sqrt{5}+5}{10}}\bigr)^{-2}=0.2306\dotsc$. Then the inclusion $\mscr{C}_{U_F^+}\subseteq \LL^{\infty}(U_F^+)$ is quasi-split and the relative commutant $\mscr{C}_{U_F^+}'\cap \LL^{\infty}(U_F^+)$ is not contained in $\mscr{C}_{U_F^+}$.
\end{corollary}
\begin{proof}
By \cite[Theorem 33]{CCAP} $\LL^{\infty}(U_F^+)$ is a type $\operatorname{III}$ factor. Therefore Proposition \ref{Prop:notmasa} applies and we know that $\mscr{C}_{U_F^+}'\cap \LL^{\infty}(U_F^+)$ is a type $\operatorname{III}$ algebra, hence cannot be contained in $\mscr{C}_{U_F^+}$, which is a finite von Neumann algebra.

An alternative argument can go as follows. By \cite[Th\'{e}or\`{e}me 1 (iii)]{BanicaUn} the character of the fundamental representation of $U_F^+$ has the same distribution (with respect to the Haar integral) as a circular variable\footnote{Recall that $x$ is circular if $x=s_1+is_2$, where $s_1$ and $s_2$ are freely independent semicircular variables.}, so we can conclude that $\mscr{C}_{U_F^+}$ is isomorphic to the free group factor $\LL(F_2)$, in particular it is a factor. If the relative commutant $\mscr{C}_{U_F^+}'\cap \LL^{\infty}(U_F^+)$ was contained in $\mscr{C}_{U_F^+}$ then the center of $\LL^{\infty}(U_F^+)$ would be contained in $\mscr{C}_{U_F^+}$, so $\LL^{\infty}(U_F^+)$ has to be a factor. What is more, if $\mscr{C}_{U_F^+}'\cap \LL^{\infty}(U_F^+) \subseteq \mscr{C}_{U_F^+}$ then $\mscr{C}_{U_F^+}'\cap \LL^{\infty}(U_F^+) = \mscr{C}_{U_F^+} \cap \mscr{C}_{U_F^+}'\cap \LL^{\infty}(U_F^+) = \CC \mathds{1}$, i.e. the inclusion is irreducible. By Remark \ref{rem:splitfactor} a quasi-split inclusion of factors is actually split and it is easy to check that a proper split inclusion cannot be irreducible.
\end{proof}
\subsection*{Acknowledgements}
MW was supported by the Research Foundation –- Flanders (FWO) through a Postdoctoral Fellowship, by long term structural funding -- Methusalem grant of the Flemish Government -- and by European Research Council Starting Grant 677120 INDEX. JK was partially supported by the NCN (National Centre of Science) grant 2014/14/E/ST1/00525. Furthermore, JK would like to express his gratitude towards Adam Skalski and Piotr M.~Sołtan for many helpful discussions. The authors would like to thank the anonymous referees for several useful remarks.

\end{document}